\documentclass[a4paper, reqno]{amsart}

\usepackage[utf8]{inputenc}
\usepackage[all]{xy}
\usepackage[english]{babel}
\usepackage{amssymb,amsmath,amsthm,amsfonts}
\usepackage[hidelinks]{hyperref}
\usepackage[alphabetic]{amsrefs}
\usepackage{csquotes}

\DeclareFontFamily{OT1}{pzc}{}
\DeclareFontShape{OT1}{pzc}{m}{it}{<-> s * [1.100] pzcmi7t}{}
\DeclareMathAlphabet{\mathpzc}{OT1}{pzc}{m}{it}
\newcommand{\G}{\mathpzc G}
\newcommand{\Sol}{\mathpzc{Sol}}
\newcommand{\ID}{\mathpzc{ID}}
\newcommand{\GP}{\mathpzc{GP}}
\newcommand{\bigslant}[2]{{\raisebox{.2em}{$#1$}\left/\raisebox{-.2em}{$#2$}\right.}}

\newcommand{\id}{\mathrm{id}}
\newcommand{\Ric}{\mathrm{Ric}}
\newcommand{\Scal}{\mathrm{Scal}}

\renewcommand{\S}{\Sigma}
\newcommand{\na}{\nabla}
\newcommand{\n}{\nabla}
\newcommand{\im}{\mathrm{im}}
\renewcommand{\L}{\mathcal{L}}

\newcommand{\Hess}{\mathrm{Hess}}
\newcommand{\g}{{\tilde g}}
\newcommand{\h}{{\tilde h}}
\newcommand{\tk}{{\tilde k}}
\newcommand{\m}{{\tilde m}}
\renewcommand{\o}{\omega}
\newcommand{\supp}{\mathrm{supp}}
\newcommand{\tr}{\mathrm{tr}}
\renewcommand{\a}{\alpha}

\renewcommand{\b}{\beta}
\renewcommand{\d}{\partial}
\newcommand{\abs}[1]{\left\lvert#1\right\rvert}				
\newcommand{\norm}[1]{\left\lVert#1\right\rVert} 			
\renewcommand{\div}{\mathrm{div}}
\newcommand{\grad}{\mathrm{grad}}
\newcommand{\D}[0]{\mathcal{D}}
\newcommand{\R}[0]{\mathbb{R}}							
\newcommand{\N}[0]{\mathbb{N}}							
\newcommand{\Z}[0]{\mathbb{Z}}							

\newcommand{\der}[2]{\frac{d}{d #1}\Big|_{#1 = #2}}

\theoremstyle{plain}
\newtheorem{thm}{Theorem}[section]
\newtheorem{prop}[thm]{Proposition}
\newtheorem{lemma}[thm]{Lemma}
\newtheorem{cor}[thm]{Corollary}
\theoremstyle{definition}
\newtheorem{definition}[thm]{Definition}

\newtheorem{remark}[thm]{Remark}
\newtheorem{example}[thm]{Example}

\author{Oliver Lindblad Petersen}
\title{On the Cauchy problem for the linearised Einstein equation}
\address{University of Hamburg, Department of Mathematics, Bundesstraße 55, 20146 Hamburg, Germany}
\email{oliver.petersen@uni-hamburg.de}

\begin{document}
	\begin{abstract}
	A classical problem in general relativity is the Cauchy problem for the linearised Einstein equation (the initial value problem for gravitational waves) on a globally hyperbolic vacuum spacetime. 
	A well-known result is that it is uniquely solvable up to gauge solutions, given initial data on a spacelike Cauchy hypersurface. 
	The solution map is an isomorphism between initial data (modulo gauge producing initial data) and solutions (modulo gauge solutions).

	 In the first part of this work, we show that the solution map is actually an isomorphism of locally convex topological vector spaces. 
	 This implies that the equivalence class of solutions depends continuously on the equivalence class of initial data. 
	 We may therefore conclude well-posedness of the Cauchy problem.
	 
	 In the second part, we show that the linearised constraint equations can always be solved on a closed manifold with vanishing scalar curvature. 
	 This generalises the classical notion of TT-tensors on flat space used to produce models of gravitational waves. 
	
	All our results are proven for smooth and distributional initial data of arbitrary real Sobolev regularity.
	\end{abstract}
	\keywords{linearised Einstein equation \and Cauchy problem \and gravitational wave \and linearised constraint equation}
	\subjclass[2010]{Primary 83C35; Secondary 35L15}
	\maketitle
	\tableofcontents
\begin{sloppypar}
	\section{Introduction}

Gravitational waves are usually modelled as solutions to the linearised Einstein equation. 
The purpose of this work is to extend well-known results on the Cauchy problem for the linearised Einstein equation. 

The classical existence theorem for the Cauchy problem for the (non-linear) Einstein equation, proven by Choquet-Bruhat in \cite{F-B1952}, can be formulated as follows. Given a Riemannian manifold $(\S, \g)$ with a smooth $(0,2)$-tensor $\tk$ satisfying the vacuum constraint equations
\[
	\Phi(\g, \tk) := \left( 
	\begin{array}{ll}
		\Scal(\g) - \g(\tk, \tk) + (\tr_\g\tk)^2 \\
		\div(\tk - \tr_\g (\tk)\g) 
	\end{array}	 \right) = 0,
\]
there is a globally hyperbolic spacetime $(M, g)$ satisfying the Einstein vacuum equation
\[
	\Ric(g) = 0,
\]
and an embedding $\iota: \S \hookrightarrow M$ such that $(\g, \tk)$ are the induced first and second fundamental forms. 
It was shown in \cite{C-BG1969} that each such globally hyperbolic development can be embedded into a \enquote{maximal globally hyperbolic development}, determined up to isometry. 
Assume now that $(M, g)$ is a smooth vacuum spacetime and let $\S \subset M$ denote a Cauchy hypersurface.
Using methods analogous to \cite{F-B1952}, it can be shown (see \cite{FewsterHunt2013}*{Thm. 3.1, Thm. 3.3} and \cite{FisherMarsden1979}*{Thm. 4.5}) that the Cauchy problem for the linearised Einstein equation can be solved. 
More precisely, given smooth $(0,2)$-tensors $(\h, \m)$ on $\S$ such that the linearised constraint equation is satisfied, i.e. 
\[
	D\Phi_{\g, \tk}(\h, \m) = 0,
\]
there is a smooth $(0,2)$-tensor $h$ on $M$ such that the linearised Einstein equation
\[
	D\Ric_g(h) = 0
\]
is satisfied. 
Analogously to the (non-linear) Einstein equation, the solution is only determined up to addition of a gauge solution. 
However, the \emph{equivalence class} of gauge solutions is uniquely determined by the corresponding \emph{equivalence class of initial data}.
In other words, the solution map
\begin{gather*}
	\bigslant{\text{Initial data on }\S}{\text{Gauge producing initial data}} \\
	\downarrow \\
	\bigslant{\text{Global solutions on }M}{\text{Gauge solutions}}
\end{gather*}
is an isomorphism. Our first main result is Theorem \ref{thm: Wellposedness}, which says that this map is an isomorphism of \emph{locally convex topological vector spaces}.
This concludes \emph{well-posedness} of the Cauchy problem for the linearised Einstein equation, meaning global existence, uniqueness and continuous dependence on initial data.
We prove this for initial data of arbitrary real Sobolev regularity.
This enables us to model gravitational waves that are very singular at a certain initial time.
See Example \ref{ex: arbitrarily irregular} for an example of arbitrarily irregular initial data that does not produce gauge solutions.

In order to apply Theorem \ref{thm: Wellposedness} in practice, it is necessary to understand the space 
\begin{equation} \label{eq: quotient}
	\bigslant{\text{Initial data on }\S}{\text{Gauge producing initial data}}.
\end{equation}
We show that this space can be well understood if we assume that $\S$ is compact and $\tk = 0$, in which case the constraint equation $\Phi(\g, \tk) = 0$ just means $\Scal(\g) = 0$. 
Using Moncrief's splitting theorem, it is easy to calculate that solutions $(\h, \m)$ of 
\begin{align}
	\Delta  \tr_\g \h &= \g(\Ric(\g), \h), \label{eq: Eq1firstFF} \\
		\div \h &= 0, \label{eq: Eq2firstFF} \\
		\Delta \tr_\g \m &= - \g(\Ric(\g), \tilde m), \label{eq: Eq1secondFF} \\
		\div (\m - (\tr_\g \m)\g) &= 0, \label{eq: Eq2secondFF}
\end{align}
are in one-to-one correspondence with elements in \eqref{eq: quotient} in case $\tk = 0$. In other words, one can show that
\begin{align*}
	\text{Solutions to (\ref{eq: Eq1firstFF} - \ref{eq: Eq2secondFF})} &\to \bigslant{\text{Initial data on }\S}{\text{Gauge producing i.d.}}\\
	(\h, \m) &\mapsto [(\h, \m)] 
\end{align*}
is an isomorphism of topological vector spaces, see Proposition \ref{prop: Moncrief_splitting}. 
Our second main result, Theorem \ref{thm: initial_data_split}, concerns solving equations \eqref{eq: Eq1firstFF} - \eqref{eq: Eq2secondFF}. 
We show that given $(0,2)$-tensors $(\a, \b)$, there is a unique decomposition 
\begin{align}
	\a &= \h + L \o + C \Ric(\g) + \phi \g, \label{eq: alpha}\\
	\b &= \m + L \eta + C' \Ric(\g) + \psi \g, \label{eq: beta}
\end{align}
where $(\h, \m)$ solves (\ref{eq: Eq1firstFF} - \ref{eq: Eq2secondFF}), $L$ is the conformal Killing operator, $\o, \eta$ are one-forms, $\phi, \psi$ are functions such that $\int_\S \phi d\mu_\g = \int_\S \psi d\mu_\g = 0$. 
If $\Ric(\g) = 0$, then the solution space of (\ref{eq: Eq1firstFF} - \ref{eq: Eq2secondFF}) is spanned by the TT-tensors and $C \g$ for any $C \in \R$ and \eqref{eq: alpha} and \eqref{eq: beta} are nothing but the usual $L^2$-split.
Note however that TT-tensors are only guaranteed to solve (\ref{eq: Eq1firstFF} - \ref{eq: Eq2secondFF}) in case $\Ric(\g) = 0$. 
Our result therefore extends the classical use of TT-tensors to produce models of gravitational waves.

We start by introducing spaces of sections of various regularity in Section \ref{sec: Notation}. In Section \ref{ch: non-linear_Cauchy} we formulate the Cauchy problem for the linearised Einstein equation. The goal of Section \ref{ch: Cauchy_lin_Ein} is then to prove our first main result, Theorem \ref{thm: Wellposedness}, concerning the linearised Einstein equation. We conclude in Section \ref{ch: lin_constraint} with our second main result, Theorem \ref{thm: initial_data_split}, concerning the linearised constraint equations. 

We expect that our results can be generalised to various models with matter, using the methods presented here, but we will for simplicity restrict to the vacuum case.

\subsubsection*{Acknowledgements}
It is a pleasure to thank my PhD supervisor Christian Bär for suggesting this topic and for many helpful comments. 
I especially want to thank Andreas Hermann for many discussions and for reading early versions of the manuscript. 
Furthermore, I would like to thank the Berlin Mathematical School, Sonderforschungsbereich 647 and Schwerpunktprogramm 2026, funded by Deutsche Forschungsgemeinschaft, for financial support.

\section{The function spaces} \label{sec: Notation}

Let us start by introducing our notation. 
All manifolds, vector bundles and metrics will be smooth, but the sections will have various regularity.
Assume that $M$ is a smooth manifold and let $E \to M$ be a real vector bundle over $M$. We denote the \emph{space of smooth sections} in $E$ by
\[
	C^\infty(M, E),
\]
equipped with the canonical Fr{\'e}chet space structure. Let us denote the space of \emph{sections of Sobolev regularity} $k \in \R$ by
\[
	H^k_{loc}(M, E).
\]
By the Sobolev embedding theorem, we may write
\[
	H^\infty_{loc}(M, E) := \bigcap_{k \in \R} H_{loc}^k(M, E) = C^\infty(M, E).
\]
We will write $C^\infty(M)$ and $H^k_{loc}(M)$ instead of $C^\infty(M, E)$ and $H^k_{loc}(M, E)$ whenever it is clear from the context what vector bundle is meant. 
For a compact subset $K \subset M$ and a $k \in \R \cup \{\infty\}$, let
\[
	H^k_K(M, E)
\]
denote the sections of Sobolev regularity $k$ with support contained in $K$.
As above, we have $H_K^\infty(M, E) = C^\infty_K(M, E)$.
Define the space of \emph{sections of compact support} of Sobolev regularity $k$ in $E$ by
\[
	H_c^k(M, E) := \bigcup_{\stackrel{K \subset M}{\text{compact}}} H^k_K(M, E).
\]
In order to define the topology, choose an exhaustion of $M$ by compact sets $K_1 \subset K_2 \subset \hdots \subset \bigcup_{n \in \N} K_n = M$.
Since $H^k_{K_n}(M) \subset H^k_{K_{n+1}}(M)$ is closed for all $n \in \N$, the strict inductive limit topology is defined on $H_c^k(M)$ (see for example \cite{Treves1967}).
A linear map $L: H_c^k(M) \to V$ into a locally convex topological vector space $V$ is continuous if and only if $L|_{H_K^k(M)}:H_K^k(M) \to V$ is continuous for any compact set $K \subset M$.
The strict inductive limit topology turns $H_c^k(M)$ into a locally convex topological vector space (in fact an LF-space) and is independent of the choice of exhaustion. The following lemma gives the notion of convergence of a sequence (or net) of sections.

\begin{lemma} \label{le: convergence_H_c}
Let $k \in \R \cup \{\infty\}$. 
Assume that $V \subset H_c^k(M)$ is bounded. 
Then there is a compact subset $K \subset M$ such that $V \subset H_K^k(M)$. 
In particular, if $u_n \to u$ is a converging sequence (or net), then there is a compact subset $K \subset M$ such that $\supp(u_n), \supp(u) \subset K$ and $u_n \to u$ in $H^k_K(M, E)$.
\end{lemma}
\begin{proof}
Assume to reach a contradiction, that the statement is not true. Let $K_1 \subset K_2 \subset \hdots$ be a exhaustion by compact subsets of $M$. 
By assumption, for each $i \in \N$ there is an $f_i \in V$ such that $\supp(f_i) \not\subset K_i$. 
Hence there are test sections $\varphi_i \in C_c^\infty(M, E^*)$ such that $\supp(\varphi_i) \subset K_i^\mathsf{c}$ and $f_i[\varphi_i] \neq 0$. 
Consider the convex subset containing zero, given by 
\[
	W := \left\{f \in H_c^k(M, E) \mid \abs{f[\varphi_i]} < \frac{\abs{f_i[\varphi_i]}}{i}, \ \forall i\right\} \subset H_c^k(M, E).
\]
We claim that $W$ is open. 
We have 
\[
	W \cap H^k_{K_j}(M, E) = \bigcap_{i = 1}^{j-1} \left\{f \in H_{K_j}^k(M, E) \mid \abs{f[\varphi_i]} < \frac{\abs{f_i[\varphi_i]}}{i} \right\}.
\]
Since $f \mapsto \abs{f[\varphi_i]}$ is a continuous function on $H_{K_j}^k(\S, E)$, this is a finite intersection of open sets and hence open. 
Hence $W$ is open. 
Note that for each $T > 0$, we have $f_i \notin T\cdot W$ if $i > T$. 
It follows that $V$ is not bounded.
\end{proof}
 
Let $E^* \to M$ be the dual vector bundle to $E$. 
We denote the space of all continuous functionals on $C_c^\infty(M, E^*)$ by $\D'(M, E)$ and we equip it with the weak*-topology. 
Elements of $\D'(M, E)$ are called \emph{distributional sections in $E$}.
For $k < 0$, the elements of $H^k_{loc}(M)$ cannot be realised as measurable functions, only as distributions. 
The natural inclusion $L^1_{loc}(M, E) \hookrightarrow \D'(M, E)$ is given by
\[
	f \mapsto \left(\varphi \mapsto \int_{M} \varphi(f) d \mu_g \right),
\]
for some fixed (semi-)Riemannian metric $g$ on $M$.
The image of the embedding
\[
	C^\infty(M, E) \hookrightarrow \mathcal D'(M, E)
\]
is dense. 
We have the continuous inclusions
\[
	H_K^k(M) \subset H_c^k(M) \subset H_{loc}^k(M) \subset \D'(M, E)
\]
for each compact set $K \subset M$ and $k \in \R \cup \{\infty\}$.
Moreover, it is a standard result that each compactly supported distribution is of some Sobolev regularity, i.e.
\[
	\D'_c(M, E) = \bigcup_{k \in \R} H_c^k(M, E).
\]

Let us explain how linear differential operators act on distributional sections. 
Since any Sobolev section is a distribution, this shows how differential operators act on Sobolev spaces as well. 
Assume that $E, F \to M$ are equipped with positive definite metrics $\langle \cdot, \cdot \rangle_E$ and $\langle\cdot, \cdot \rangle_F$.
Denote the space of linear differential operators of order $m \in \N$ mapping sections in $E$ to sections in $F$ by $\mathrm{Diff}_m(E, F)$.
Given a $P \in \mathrm{Diff}_m(E, F)$, define the formal adjoint operator $P^* \in \mathrm{Diff}_m(F, E)$ to be the unique differential operator such that 
\begin{equation} \label{eq: int_local_int}
	\int_{M} \langle P\varphi, \psi \rangle_F = \int_{M} \langle \varphi, P^* \psi \rangle_E d\mu_g,
\end{equation}
for all $\psi \in C_c^\infty(M, F)$ and $\varphi \in C^\infty_c(M, E)$. Using this, $P$ can be extended to act on distributions by the formula
\[
	PT[\langle \cdot, \psi\rangle_F] = T[\langle \cdot, P^* \psi\rangle_E].
\]
This coincides with equation \eqref{eq: int_local_int} when $T$ can be identified with a compactly supported smooth section. 
$P$ extends to continuous maps
\begin{align*}
	\D'(M, E) &\to \mathcal D'(M, F), \\
	H^k_{loc}(M, E) &\to H^{k-m}_{loc}(M, F), \\
	H^k_K(M, E) &\to H^{k-m}_K(M, F), \\
	H^k_c(M, E) &\to H^{k-m}_c(M, F),
\end{align*}
for all $k \in \R \cup \{\infty\}$ and all compact subsets $K \subset M$.
The following lemma will be of importance.
\begin{lemma} \label{le: two topologies}
Let $k \in \R \cup \{\infty\}$ and let $P \in \mathrm{Diff}_m(E, F)$. Then the induced subspace topology on 
\[
	H^k_c(M, E) \cap \ker(P)
\]
is the same as the strict inductive limit topology induced by the embeddings 
\[
	H^k_K(M, E) \cap \ker(P) \hookrightarrow H^k_c(M, E) \cap \ker(P).
\]
\end{lemma}
\begin{proof}
Let $u_n \to u$ be a net converging in $H^k_c(M, E) \cap \ker(P)$ with respect to the subspace topology. Then $u_n \to u$ in $H^k_c(M, E)$, which by Lemma \ref{le: convergence_H_c} means that there is a compact subset $K \subset M$ such that $u_n \to u$ in $H^k_K(M, E)$. It follows that $u_n \to u$ in $H_K^k(M, E) \cap \ker(P)$ and hence in $H_c^k(M, E) \cap \ker(P)$ with respect to the strict inductive limit topology, since the embedding is continuous. The other direction is clear.
\end{proof}

Assume now that $(M,g)$ is a smooth globally hyperbolic spacetime. By \cite{BernalSanches08}*{Thm. 1.1} there is a Cauchy temporal function $t:M \to \R$, i.e. for all $\tau \in t(M)$, $\S_{\tau} := t^{-1}(\tau)$ is a smooth spacelike Cauchy hypersurface and $\grad(t)$ is timelike and past directed. The metric can then be written as
\[
	g = -\a^2 dt^2 + \g_t,
\]
where $\a: M \to \R$ is a positive function and $\g_\tau$ denotes a Riemannian metric on $\S_\tau$, depending smoothly on $\tau \in t(M)$. 
It follows that the future pointing unit normal $\nu$ is given by $\nu = -\frac1 \a \grad(t)|_{\S_\tau}$. 
Let us use the notation
\[
	\na_{t} := \na_{\grad(t)}.
\]
For each $k \in \R$, we get a Fr{\'e}chet vector bundle 
\[
	(H_{loc}^k(\S_\tau, E|_{\S_\tau}))_{\tau \in t(M)}.
\]
We denote the $C^m$-sections in this vector bundle by
\[
	C^m\left(t(M), H_{loc}^k(\S_\cdot, E|_{\S_\cdot}) \right).
\]
This is a Fr{\'e}chet space. When solving wave equations, the solutions typically lie in the following spaces of sections of \emph{finite energy of infinite order}:
\[
	CH^k_{loc}(M, E, t) := \bigcap_{j = 0}^\infty C^j\left(t(M), H_{loc}^{k-j}(\S_\cdot, E|_{\S_\cdot})\right).
\]
The spaces $CH_{loc}^k(M, E, t)$ carry a natural induced Fr{\'e}chet topology. For $k = \infty$, write
\[
	CH^\infty_{loc}(M, E, t) := \bigcap_{k \in \R} CH^k_{loc}(M, E, t) = C^\infty(M, E).
\]
Note that we have the continuous embedding
\begin{equation}
	CH^k_{loc}(M, E, t) \hookrightarrow H^{\lfloor k \rfloor}_{loc}(M, E), \label{eq: CH_H-embedding}
\end{equation}
where $\lfloor k \rfloor$ is the largest integer smaller than or equal to $k$. The finite energy sections can be considered as distributions defined by
\[
	u[\varphi] := \int_{t(M)}u(\tau) \left[(\a \varphi)|_{\S_\tau} \right] d\tau.
\]
For any subset $A \subset M$, let $J^{-/+}(A)$ denote the causal past/future of $A$ and denote their union $J(A)$. 
A subset $A \subset M$ is called \emph{spatially compact} if $A \subset J(K)$ for some compact subset $K \subset M$. 
For each spatially compact subset $A \subset M$, the space
\[
	CH_{A}^k(M, E, t) := \{f \in CH_{loc}^k(M, E, t) \mid \supp(f) \subset A \} \subset CH_{loc}^k(M, E, t)
\]
is closed and therefore also a Fr{\'e}chet space.  We define the \emph{finite energy sections of spatially compact support} by
\[
	CH_{sc}^k(M, E, t) := \bigcup_{\stackrel{A}{\text{spatially compact}}} CH_{A}^k(M, E, t)
\] 
with the strict inductive limit topology. 
The strict inductive limit topology is defined, since if $K_i$ is an exhaustion of a Cauchy hypersurface $\S$, then $J(K_i)$ is an exhaustion of $M$ by spatially compact sets. 
Similar to before, the notion of convergence is given by the following lemma. 
The proof is analogous to the proof of Lemma \ref{le: convergence_H_c}. 
\begin{lemma} \label{le: convergence_CH_sc}
Assume that $V \subset CH_{sc}^k(M, E, t)$ is bounded. 
Then there is a compact subset $K \subset \S$ such that $V \subset CH_{J(K)}^k(M, E, t)$. 
In particular, if $u_n \to u$ is a converging sequence (or net), then there is a compact subset $K \subset \S$ such that $\supp(u_n), \supp(u) \subset J(K)$ and $u_n \to u$ in $CH^k_{J(K)}(M, E, t)$.
\end{lemma}

Any $P \in \mathrm{Diff}_m(E, F)$ extends to continuous maps
\begin{align*}
	CH^k_{loc}(M, E, t) &\rightarrow CH^{k-m}_{loc}(M, E, t), \\
	CH^k_{A}(M, E, t) &\rightarrow CH^{k-m}_{A}(M, E, t), \\
	CH^k_{sc}(M, E, t) &\rightarrow CH^{k-m}_{sc}(M, E, t), 
\end{align*}
for any $k \in \R \cup \{\infty\}$ and any spatially compact set $A \subset M$. 
The following lemma is proven analogously to Lemma \ref{le: two topologies}, using Lemma \ref{le: convergence_CH_sc} instead of Lemma \ref{le: convergence_H_c}.
\begin{lemma} \label{le: two topologies2}
Let $k \in \R \cup \{\infty\}$ and let $P \in \mathrm{Diff}_m(E, F)$. Then the induced subspace topology on 
\[
	CH^k_{sc}(M, E) \cap \ker(P)
\]
is the same as the strict inductive limit topology induced by the embeddings 
\[
	CH^k_{J(K)}(M, E) \cap \ker(P) \hookrightarrow CH^k_{sc}(M, E) \cap \ker(P).
\]
\end{lemma}

Since we will commonly work with distributional tensors, let us conclude this section by showing how some standard tensor operations are made on distributional tensors. Let $g$ be a smooth semi-Riemannian metric on a manifold $M$, extended to tensor fields.
\begin{itemize}
\item If $X \in \mathcal D'(M, TM)$ and $Y \in C^\infty(M, TM)$, then the distribution $g(X, Y)$  is given by 
\[
	g(X, Y)[\varphi] = X[\varphi g(\cdot, Y)].
\]
This is well-defined since $\varphi g(\cdot, Y) \in C_c^\infty(M, T^*M)$. Using this, we can project $X$ to vector subbundles for example.
\item Similarly, if $a \in \mathcal D'(M, T^*M \otimes T^*M)$ and $b \in C^\infty(M, T^*M \otimes T^*M)$, then the distribution $g(a, b)$ is defined by
\[
	g(a, b)[\varphi] := a[\varphi g(\cdot, b)].
\]
In particular, the trace of $a$ with respect to $g$ is defined and equals
\[
	\tr_g(a) := g(g, a).
\]
\end{itemize}

\section{Linearising the Einstein equation} \label{ch: non-linear_Cauchy}

We will study the linearisation of the vacuum Einstein equation
\[
	\Ric(g) = 0
\]
globally hyperbolic spacetimes of dimension at least $3$.
Recall that if $(M, g)$ is a vacuum spacetime (i.e. $\Ric(g) = 0$) and $\S \subset M$ is a spacelike hypersurface, then the induced first and second fundamental forms $(\g, \tk)$ on $\S$ satisfy
\begin{align}
	\Scal(\g) + (\tr_\g \tk)^2 - \g(\tk, \tk) =&0, \label{eq: ham_constraint} \\
	\div (\tk - (\tr_{\g} \tk) \g) =& 0. \label{eq: momentum_constraint}
\end{align}
A famous result by Choquet-Bruhat and Geroch gives a converse statement to this.
\begin{thm}[\cite{C-BG1969}, \cite{F-B1952}] \label{thm: Choquet-Bruhat}
Given a Riemannian manifold $(\S, \g)$ and a smooth $(0,2)$-tensor $\tk$ on $\S$ satisfying \eqref{eq: ham_constraint} and \eqref{eq: momentum_constraint}, there is a maximal globally hyperbolic development $(M, g)$ of $(\S, \g, \tk)$ that is unique up to isometry. 
\end{thm}
A globally hyperbolic development means that there exists an embedding $\iota:\S \hookrightarrow M$ such that $\iota(\S) \subset M$ is a Cauchy hypersurface and $(\g, \tk)$ are the induced first and second fundamental form. 
In particular, $(M, g)$ is a globally hyperbolic spacetime. 
That $(M, g)$ is maximal means that any other globally hyperbolic development embeds isometrically into $(M, g)$ such that the embedding of the Cauchy hypersurface is respected.

Note that a maximal globally hyperbolic development can of course only be unique up to isometry. We will see that this \enquote{gauge invariance} shows up in the linearised case as an important feature of the linearised Einstein equation.

\subsection{The linearised Einstein equation}
\label{sec: Linearisation}

Assume in the rest of the paper, unless otherwise stated, that $(M,g)$ is a globally hyperbolic spacetime of dimension at least $3$ satisfying the \emph{Einstein equation}, i.e.
\[
	\Ric(g) = 0.
\]
We \emph{do not} require $(M,g)$ to be maximal in the sense of Theorem \ref{thm: Choquet-Bruhat}. 
Let us now linearise the Einstein equation around $g$. 
For this, we first define the Lichnerowicz operator
\[
 \Box_L h := \na^*\na h - 2 \mathring{R}h,
\]
where 
\begin{align*}
	\na^*\na &:= - \tr_g(\na^2), \qquad \qquad \text{(connection-Laplace operator)}\\
	\mathring{R}h(X,Y) &:= \tr_g(h(R(\cdot, X)Y, \cdot)),
\end{align*}
for any $(0,2)$-tensor $h$ and $X, Y \in TM$. It will be natural to write the linearised Einstein equation as the Lichnerowicz operator plus a certain Lie derivative of the metric $g$. We use the following notation
\begin{align*}
	\n \cdot h(X) &:= \tr_g(\n_{\cdot}h(\cdot, X)), \qquad\text{(divergence)} \\
	\overline h &:= h - \frac12 \tr_g(h)g. 
\end{align*}
for any $h \in C^\infty(M, S^2M)$ and $X \in TM$.

\begin{lemma}[Linearising the Ricci curvature] \label{le: LinearisedEinstein}
Any curve of smooth Lorentz metrics $g_s$, such that $g_0 = g$ satisfies
\[
	\der s0 \Ric(g_s) = \frac12 \left( \Box_L h + \L_{(\na \cdot \overline h)^\sharp} g\right),
\]
where
\[
	h := \der s0 g_s
\]
and $\L$ is the Lie derivative and $\sharp$ is the musical isomorphism (\enquote{raising an index}). 
\end{lemma}
\begin{proof}
This straightforward computation can for example be found in \cite{Besse1987}*{Thm. 1.174}.
\end{proof}

Let us denote the vector bundle of symmetric $2$-tensors on $M$ by
\[
	S^2M := \otimes_{sym}^2T^*M.
\]
Lemma \ref{le: LinearisedEinstein} motivates the following definition. 

\begin{definition}[The linearised Einstein equation]
	 We define the linearised Ricci curvature
	 \[
	 	D \Ric(h) := \frac12 \left( \Box_L h + \L_{(\na \cdot \overline h)^\sharp}g \right)
	 \]
	 for any $h \in \mathcal D'(M, S^{2}M)$. We say that $h$ satisfies the \emph{linearised Einstein equation} if 
	\[
		D \Ric(h) = 0.
	\]
\end{definition}

Here we use the definition that extends to distributions, 
\[
	\L_Vg(X,Y) = g(\na_XV, Y) + g(\na_YV, X).
\] 

\begin{remark}
	Note that $\square_L$ is a wave operator, but $D\Ric$ is not (c.f. Section \ref{sec: Waves}).
\end{remark}

There are certain solutions of the linearised Einstein equation, called \enquote{gauge solutions}, which are due to \enquote{infinitesimal isometries}.

\begin{lemma}[Gauge invariance of the linearised Einstein equation]
For any vector field $V \in \mathcal D'(M, TM)$, we have 
\[
	D\Ric(\L_V g) = 0.
\]
\end{lemma}
\begin{proof}
Let us first restrict to smooth objects. Let $\varphi_s: M \to M$ be a curve of diffeomorphisms such that $\varphi_0 = \id$ and such that $\frac d {ds} \big |_{s = 0} \varphi_s = V$. Differentiating the equation
\[
	0 = \varphi_s^*\Ric(g) = \Ric(\varphi_s^* g)
\]
gives
\[
	D\Ric(\L_Vg) = 0.
\]
By density of smooth sections in distributional sections, the result extends to the general case.
\end{proof}

\subsection{The linearised constraint equation}

Assume throughout the rest of the paper, unless otherwise stated, that $\S \subset M$ is a smooth spacelike Cauchy hypersurface with future pointing unit normal vector field $\nu$. Let $(\g, \tk)$ denote the first and second fundamental forms. As mentioned in the beginning of this section, $(\g, \tk)$ will satisfy the \emph{constraint equation}
\[
	\Phi(\g, \tk) := \left( 
	\begin{array}{ll}
			\Phi_1(\g, \tk) \\
			\Phi_2(\g, \tk)
	\end{array}	 \right) = 0,
\]
where
\begin{align*}			
			\Phi_1(\g, \tk) &= \Scal(\g) - \g(\tk, \tk) + (\tr_\g\tk)^2, \\
			\Phi_2(\g, \tk) &= \tilde \na \cdot \tk - d(\tr_\g\tk),
\end{align*}
and $\tilde \na$ is the Levi-Civita connection on $\S$ with respect to $\g$. We linearise the constraint equation around $(\g, \tk)$, analogously to Lemma \ref{le: LinearisedEinstein}.

\begin{definition}[The linearised constraint equations] \label{def: lin_constr_eq}
A pair of tensors $(\h, \m) \in \mathcal D'(\S, S^{2}\S) \times \mathcal D'(\S, S^2\S)$ is said to satisfy the \emph{linearised constraint equation}, linearised around $(\g, \tk)$, if
\[ 
	D\Phi(\h, \m) := \left( \begin{array}{ll}
			D\Phi_1(\h, \m) \\
			D\Phi_2(\h, \m)
	\end{array} \right) = 0,
\]
in $\mathcal D'\left(\S, \R) \times \mathcal D'(\S, T^*\S \right)$, where
\begin{align}
 D\Phi_1(\h, \m) 
	:= & \tilde \na \cdot(\tilde \na \cdot \h - d \tr_\g \h ) - \g(\Ric(\g), \h) \nonumber \\
 	& + 2 \g(\tk \circ \tk - (\tr_\g \tk) \tk, \h) - 2 \g(\tk, \m - (\tr_\g \m)\g), \label{eq: LinConstraints1}\\
D\Phi_2(\h, \m)(X) \nonumber
	:=& - \g(\h, \tilde \na_{(\cdot)} \tk(\cdot, X)) - \g \left(\tk(\cdot, X), \tilde \na \cdot (\h - \frac12 (\tr_{\g}\h) \g)\right) \nonumber\\
	& - \frac 12 \g(\tk, \tilde \na_X \h) + d(\g(\tk, \h))(X) + \tilde \na \cdot (\m - (\tr_\g \m)\g)(X), \label{eq: LinConstraints2}
\end{align}
for any $X \in T\S$, where $\tk \circ \tk(X, Y) := \g(\tk(X, \cdot), \tk(Y, \cdot))$ for any $X, Y \in T\S$.
\end{definition}

Similarly to the non-linear case, we will be given initial data satisfying the linearised constraint equations and require a solution to induce these initial data as linearised first and second fundamental forms. Therefore, we need to linearise the following expressions 
\begin{align*}
	\g(X,Y) :=& g(X,Y), \\
	\tk(X,Y) :=& g(\na_X \nu, Y),
\end{align*}
analogously to Lemma \ref{le: LinearisedEinstein}. In order to make sense of the restriction of distributional tensors to $\S$, we assume the following regularity. 

\begin{definition}[Linearised first and second fundamental forms] \label{def: LinearisedFundamentalForms} \
Given $h \in CH_{loc}^k(M, S^2M, t)$, we define $(\h, \m) \in H_{loc}^{k}(\S, S^2\S) \times H_{loc}^{k-1}(\S, S^2\S)$ as
   \begin{align*}
   \h(X,Y) &= h(X,Y), \\
   \m (X,Y) &= -\frac12 h(\nu, \nu)\tk(X,Y) - \frac12 \na_X h(\nu, Y) - \frac12 \na_Y h(\nu, X) + \frac12 \na_\nu h(X,Y),
\end{align*}
for any $X, Y \in T\S$. We call $\h$ and $\m$ the \emph{linearised first and second fundamental forms} induced by $h$.
\end{definition}

Analogously to the non-linear case, one shows that if $\h$ and $\m$ are the linearised first and second fundamental forms induced by $h$, then using that $\Ric(g) = 0$ we get
\begin{align}
	\tr_g(D\Ric(h)) + 2 D\Ric(h)(\nu, \nu) = D\Phi_1(\h, \m), \label{eq: ricciNormalNormal}\\
	D\Ric(h)(\nu, \cdot) = D\Phi_2(\h, \m) \label{eq: ricciNormalDot}.
\end{align}
In particular, if $D\Ric(h) = 0$, the induced initial data $(\h, \m)$ must satisfy $D\Phi(\h, \m) = 0$. 
Let us now formulate the Cauchy problem of the linearised Einstein equation.

\begin{definition}[The Cauchy problem]
Let $(\h, \m) \in H_{loc}^{k}(\S, S^2\S) \times H_{loc}^{k-1}(\S, S^2\S)$ satisfy $D\Phi(\h, \m) = 0$. If $h \in CH_{loc}^k(M, S^2M, t)$ satisfies
\[
	D\Ric(h) = 0
\]
and induces $(\h, \m)$ as linearised first and second fundamental forms, then we call $h$ a \emph{solution to the Cauchy problem of the linearised Einstein equation} with initial data $(\h, \m)$. 
\end{definition}

\section{Well-posedness of the Cauchy problem} \label{ch: Cauchy_lin_Ein}

The goal of this section is to prove our first main result, Theorem \ref{thm: Wellposedness}. 
Recall the setting. 
We assume that $(M, g)$ is a globally hyperbolic spacetime of dimension at least $3$ solving the Einstein equation
\[
	\Ric(g) = 0.
\]
We also assume that $\S \subset M$ is a spacelike Cauchy hypersurface. It follows that
\[
	\Phi(\g, \tk) = 0,
\]
where $(\g, \tk)$ are the induced first and second fundamental forms.

\subsection{Existence of solution} \label{sec: Existence}

We start by proving that given initial data satisfying the linearised constraint equation, there is a solution to the linearised Einstein equation. 
The basic method is well-known and is analogous to the proof of the classical existence result for the non-linear Einstein equation \cite{F-B1952}. 
The crucial point in the proof is to translate the initial data to initial data for a wave equation. 
We show that the existence result extends to initial data of arbitrary real Sobolev degree.
Recall the notation
\[
	\bar h := h-\frac12 \tr_g(h) g.
\]

\begin{lemma} \label{le: GaugeChoice}
	For $k \in \R$, let $(\h, \m) \in H_{loc}^{k}(\S, S^2\S) \times H_{loc}^{k-1}(\S, S^2\S)$ . Assume that ${h \in CH_{loc}^k(M,S^{2}M, t)}$ satisfies
	\begin{align*}
		h(X,Y) &= \h(X,Y),  &		\na_\nu h(X,Y) &= 2 \m(X,Y) - (\h \circ \tk + \tk \circ \h)(X,Y), \\
		h(\nu, X) &= 0, &	\na_\nu h(\nu, X) &= \tilde \na \cdot \left( \h - \frac12 (\tr_\g \h)\g \right)(X), \\
		h(\nu, \nu) &= 0, &	\na_\nu h(\nu, \nu) &= -2 \tr_\g \m,
	\end{align*}
	for all $X,Y \in T\S$, where $\h \circ \tk(X, Y) := g(\h(X, \cdot), \h(Y, \cdot))$ for all $X, Y \in T\S$. Then $\h, \m$ are the first and second linearised fundamental forms induced by $h$ and
	\[
		\na \cdot \overline h|_\S = 0.
	\]
\end{lemma}

The proof is a simple computation. Let us now state the existence theorem. 

\begin{thm}[Existence of solution] \label{thm: Existence} 
Let $k \in \R \cup \{\infty\}$ and assume that $(\h, \m) \in H_{loc}^{k}(\S, S^2\S) \times H_{loc}^{k-1}(\S, S^2\S)$ satisfies
\[
	D\Phi(\h, \m) = 0.
\]
Then there exists a unique
\[
	h \in CH_{loc}^k(M,S^2M, t),
\]
inducing linearised first and second fundamental forms $(\h, \m)$, such that $h|_\S$ and $\na_\nu h|_\S$ are as in Lemma \ref{le: GaugeChoice} and
\begin{align*}
	\square_L h =& 0, \\
	\na \cdot \overline h =& 0.
\end{align*}
In particular
\[
	D\Ric (h) = 0.
\]
Moreover
\begin{equation} \label{eq: FiniteSpeed}
	\supp(h) \subset J \left(\supp(\h) \cup \supp(\m) \right).
\end{equation}
\end{thm}
From equation \eqref{eq: CH_H-embedding}, we conclude that in fact $h \in H_{loc}^{\lfloor{k}\rfloor}(M)$. 

\begin{remark} \label{rmk: add_gauge}
The property \eqref{eq: FiniteSpeed} is called \emph{finite speed of propagation}. If the initial data are compactly supported, the solution will have spatially compact support. Note however that \eqref{eq: FiniteSpeed} will not hold for all solutions with initial data $(\h, \m)$. If for example $V \in C^\infty(M, TM)$ with support not intersecting $\S$, then $h + \L_Vg$ is going to be a solution with the same initial data. The support of $\L_Vg$ needs not be contained in $ J \left(\supp(\h) \cup \supp(\m) \right)$.
\end{remark}

Using Theorem \ref{thm: Existence}, we get the following stability result. 

\begin{cor}[Stable dependence on intial data] \label{cor: Stability}
	For $k \in \R \cup \{\infty\}$, assume that $(\h_i, \m_i)_{i \in \N} \in H_{loc}^{k}(\S) \times H_{loc}^{k-1}(\S)$ such that $D\Phi(\h_i, \m_i) = 0$ and 
	\[
		(\h_i, \m_i) \to (\h, \m) \in H^{k}_{loc}(\S) \times H_{loc}^{k-1}(\S)
	\]
	in $H_{loc}^{k}(\S) \times H_{loc}^{k-1}(\S)$. Then there exists a solution $h \in CH_{loc}^k(M,t)$ inducing initial data $(\h,\m)$ and a sequence of solutions $h_i \in CH^k(M,t)$, inducing $(\h_i, \m_i)$ as initial data, such that
	\[
		h_i \to h
	\]
	in $CH_{loc}^k(M,t)$ and $\n \cdot \overline h_i = 0$.
\end{cor}
\begin{proof}
Since 
\[
	(\h_i, \m_i) \to (\h, \m),
\]
the equations in Lemma \ref{le: GaugeChoice} imply that $(h_i|_\S, \na_\nu h_i|_\S) \to (h|_\S, \na_\nu h|_\S)$. Since
\[
	\Box_L h = \Box_L h_i = 0,
\]
we conclude by continuous dependence on initial data for linear wave equations (see Corollary \ref{cor: cont_dep_id}) that $h_i \to h$.
\end{proof}

It is important to note that given converging initial data, the previous corollary gives \emph{one} sequence of converging solutions, inducing the correct initial data. Not every sequence of solutions that induce the correct initial data will converge. One could just add a gauge solution similar to Remark \ref{rmk: add_gauge}. This is the reason why the question of continuous dependence on initial data a priori does not make sense. This will be solved in Section \ref{sec: Main}, by considering \emph{equivalence classes} of solutions. Let us now turn to the proof of the theorem. 

\begin{lemma} \label{le: DivergenceFree}
If $h \in \mathcal D'(M, S^{2}M)$, then
\[
	\na \cdot \left( D\Ric(h) - \frac12 \tr_g(D\Ric(h)) g \right) = 0.
\]
\end{lemma}
\begin{proof}
For any Lorentzian metric $\hat g$, 
\[
	\hat \na \cdot \left( \Ric(\hat g) - \frac12 \tr_{\hat g} (\Ric(\hat g)) \hat g\right) = 0,
\]
where $\hat \na$ is the Levi-Civita connection with respect to $\hat g$. Linearising this equation around $g$, using $\Ric(g) = 0$, gives the equation for smooth $h$. Since the smooth sections are dense in the distributional sections, this proves the lemma.
\end{proof}

A calculation that will be very useful on many places is the following.

\begin{lemma}
Assume that $(N, \hat g)$ is a semi-Riemannian manifold with Levi-Civita connection $\hat \n$. Then
\begin{equation}
	\hat \n \cdot \left( \L_V \hat g - \frac{1}{2}\tr_{\hat g}(\L_V {\hat g})\hat g\right) = - \hat \n^* \hat \n V^\flat + \Ric(\hat g)(V, \cdot). \label{eq: Killing_wave}
\end{equation}
\end{lemma}
\begin{proof}
Let $(e_1, \hdots, e_n)$ be a local orthonormal frame with respect to $\hat g$ and define $\epsilon_i := g(e_i, e_i) \in \{-1, 1\}$. We have
\begin{align*}
	 \hat \n \cdot &\left( \L_V \hat g - \frac{1}{2}\tr_{\hat g}(\L_V {\hat g})\hat g\right)(X) \\
		&= \sum_{i = 1}^n \epsilon_i \left( \hat \na_{e_i}\L_V \hat g(e_i, X) - \d_X\hat g(\hat \na_{e_i}V, e_i) \right) \\
		&= \sum_{i = 1}^n \epsilon_i \left( \hat g(\hat \n^2_{e_i, e_i} V, X) + \hat g(\hat \n^2_{e_i, X} V, e_i) - \hat g(\hat \na^2_{X, e_i}V, e_i) \right) \\
		&= - \hat \n^* \hat \n V^\flat(X) + \Ric(\hat g)(V, X).
\end{align*}
\end{proof}

\begin{proof}[Proof of Theorem \ref{thm: Existence}]
Consider the Cauchy problem
\begin{equation}	\label{eq: L_d'Alembert}
	\Box_L h = 0 \\
\end{equation}
with $h|_\S$ and $\na_\nu h|_\S$ defined as in Lemma \ref{le: GaugeChoice}, using $(\h, \m)$. 
One checks that $(h|_\S, \na_\nu h|_\S) \in H^k_{loc}(\S, S^2M|_\S) \times H^{k-1}_{loc}(\S, S^2M|_\S)$. 
By Theorem  \ref{thm: WellposednessLinearWaves} there is a unique solution $h \in CH_{loc}^k(M, S^2M, t)$ to this Cauchy problem. 
Moreover, it follows that $\supp(h) \subset J(\supp(\h) \cup \supp(\m))$. 
We claim that $\na \cdot \overline h = 0$. Since $h \in CH_{loc}^k(M, t)$, it follows by Section \ref{sec: Notation} that $\na \cdot \overline h \in CH^{k-1}_{loc}(M, t)$. 
Lemma \ref{le: DivergenceFree} implies that
\begin{align*}
	0 	=& \na \cdot \left(D\Ric(h) - \frac12 \tr_g(D\Ric(h)) g \right) \\
		=& \frac12 \na \cdot \left(\L_{(\na \cdot \overline h)^\sharp}g - \frac12 \tr_g\left(\L_{(\na \cdot \overline h)^\sharp}g \right) g \right) \\
		\stackrel{\eqref{eq: Killing_wave}}{=}& - \frac12 \na^*\na (\na \cdot \overline h),
\end{align*}
since $\Ric(g) = 0$. 
From Lemma \ref{le: GaugeChoice}, we know that $\na \cdot \overline h|_\S = 0$. 
We now use the assumption that $D\Phi(\h, \m) = 0$ to show that $\na_\nu (\na \cdot \overline h)|_\S = 0$. 
Since we know that $\Box_L h = 0$ and $\na \cdot \overline{h}|_\S = 0$, equations \eqref{eq: ricciNormalNormal} and \eqref{eq: ricciNormalDot} imply that
\begin{align*}
	0 	&= D\Phi_1(\h, \m) \\
		&= \tr_g(D\Ric(h)) + 2 D\Ric(h)(\nu, \nu) \\
		&= \frac12 \left( \tr_g(\L_{(\na \cdot \overline h)^\sharp}g) + 2\L_{(\na \cdot \overline h)^\sharp}g(\nu, \nu) \right) \\
		&= \na_\nu (\na \cdot \overline{h})(\nu), \\
	0 	&= D\Phi_2(\h, \m)(X) \\
		&= D\Ric(h)(\nu, X) \\
		&= \frac12 \na_\nu(\na \cdot \overline{h})(X),
\end{align*}
for each $X \in T\S$. 
Altogether we have shown that $\na \cdot \overline h \in CH_{loc}^{k-1}(M,t)$ satisfies
\begin{align*}
	\na^*\na (\na \cdot \overline h) &= 0, \\
		\na \cdot \overline h|_\S &= 0, \\
		\na_\nu (\na \cdot \overline h)|_\S &= 0.
\end{align*}
Theorem \ref{thm: WellposednessLinearWaves} now implies that $\na \cdot \overline h = 0$. 
This finishes the proof.
\end{proof}

\subsection{Uniqueness up to gauge} \label{sec: Uniqueness}

We continue by showing that the solution is unique up to addition of a gauge solution.

\begin{thm}[Uniqueness up to gauge] \label{thm: Uniqueness}
Let $k \in \R \cup \{\infty\}$. Assume that $h \in CH_{loc}^k(M, S^{2}M, t)$ satisfies
\[
	D\Ric(h) = 0
\]
and that the induced first and second linearised fundamental forms vanish. Then there exists a vector field $V \in CH_{loc}^{k+1}(M, TM, t)$ such that
\[
	h = \L_V g.
\]
If $\supp(h) \subset J(K)$ for some compact $K \subset \S$, we can choose $V$ such that $\supp(V) \subset J(K)$. 
\end{thm}

We start by proving a technical lemma which is reminiscent of elliptic regularity theory. 
The difference is that we work with finite energy spaces and not Sobolev spaces.

\begin{lemma} \label{le: regularity_lie_derivative}
Let $V \in CH_{loc}^k(M, TM, t)$ with $\L_V g \in CH^k_{loc}(M, S^2M, t)$. Then $V \in CH_{loc}^{k+1}(M, TM, t)$.
\end{lemma}

\begin{proof}[Proof of Lemma \ref{le: regularity_lie_derivative}]
By assumption, 
\[
	\n^{j}_{t, \hdots, t} V \in C^0(t(M), H^{k - j}_{loc}(\S_\cdot))
\]
for all integers $j \geq 0$. 
We would be done if we could show that $\n^{j}_{t, \hdots, t} V \in C^0(t(M), H^{k - j+1}_{loc}(\S_\cdot))$ for all integers $j \geq 0$. 
By commuting derivatives, note that
\begin{align*}
		\L_{\n^{j}_{t, \hdots, t} V} g(X, Y) &= g(\n_X\n^{j}_{t, \hdots, t} V, Y) + g(\n_Y \n^{j}_{t, \hdots, t}V, X) \\
		&= (\n_t)^j\L_Vg(X, Y) + P_j(V)(X, Y),
\end{align*}
where $P_j$ is some differential operator of order $j$. 
Using the assumptions, this shows that
\[
	\L_{\n^{j}_{t, \hdots, t} V} g \in CH_{loc}^{k-j}(M, t).
\]
For each $\tau \in t(M)$, let $(\g_\tau, \tk_\tau)$ be the induced first and second fundamental forms on the Cauchy hypersurface $\S_\tau$.
Let $\n^{j}_{t, \hdots, t} V|_{\S_\tau} =: (\n^{j}_{t, \hdots, t} V)^\perp|_{\S_\tau} \nu_\tau + (\n^{j}_{t, \hdots, t} V)^\parallel|_{\S_\tau}$ be the projection onto parallel and normal components with respect to $\S_\tau$, where $\nu_\tau$ is the future pointing normal vector field along $\S_\tau$. Using this, we get a split $TM|_{\S_\tau} \cong \R \oplus T\S_\tau$. Note that 
\[
	\L_{(\n^{j}_{t, \hdots, t} V)^\parallel|_{\S_\tau}} \g_\tau = \L_{(\n^{j}_{t, \hdots, t} V)^\parallel}g|_{\S_\tau} - 2 (\n^{j}_{t, \hdots, t} V)^\perp|_{\S_\tau} \tk_\tau.
\]
It follows that 
\[
	\tau \mapsto \L_{(\n^{j}_{t, \hdots, t} V)^\parallel|_{\S_\tau}} \g_\tau \in CH_{loc}^{k-j}(M, t) \subset C^0(t(M), H_{loc}^{k-j}(\S_\cdot)).
\]
Since 
\[
	X \mapsto \L_X \g_{\tau} \in \mathrm{Diff}_1(T\S_\tau, S^2\S_\tau)
\]
is a differential operator of injective principal symbol, elliptic regularity theory implies that $(\n^{j}_{t, \hdots, t} V)^\parallel \in C^0(t(M), H_{loc}^{k+1-j}(\S_\cdot))$ for all integers $j \geq 0$. Using this, we conclude that
\begin{align*}
	\d_X ((\n^{j}_{t, \hdots, t} V)^\perp) &= - \a g(\n_X \n^{j}_{t, \hdots, t} V, \grad(t)) + g((\n^{j}_{t, \hdots, t} V)^\parallel, \n_X(\a\grad(t))) \\
		&= -\a \L_{\n^{j}_{t, \hdots, t}V}g (\grad(t), X) + \a g(\n_t \n^{j}_{t, \hdots, t} V, X) \\
		& \quad + g((\n^{j}_{t, \hdots, t} V)^\parallel, \n_X(\a\grad(t))) \\
		&= - \a \n^{j}_{t, \hdots, t} \L_Vg (\grad(t), X) + \a g((\n^{j+1}_{t, \hdots, t} V)^\parallel, X) \\
		& \quad + Q_j(V)(X) \in C^{0}(t(M), H_{loc}^{k-j}(\S_\cdot))
\end{align*}
for all $X \in T\S_\cdot$, since $Q_j$ is some differential operator of order $j$. We conclude that
\[
	d((\n^{j}_{t, \hdots, t} V)^\perp) \in C^{0}(t(M), H_{loc}^{k-j}(\S_\cdot, T^*\S_\cdot)).
\]
Since $d$ is a first order linear differential operator mapping functions to one-forms on $\S_\tau$ and its principal symbol is injective, we conclude that $(\n^{j}_{t, \hdots, t} V)^\perp \in C^0(t(M), H_{loc}^{k+1-j}(\S_\cdot))$ for all integers $j \geq 0$. We conclude that 
\[
	\n^{j}_{t, \hdots, t} V \in  C^0(t(M), H_{loc}^{k+1-j}(\S_\cdot))
\]
for all integers $j \geq 0$, which is the same as $V \in CH_{loc}^{k+1}(M, TM, t)$.
\end{proof}

The proof of the Theorem \ref{thm: Uniqueness} is a generalisation of the proof of \cite{FewsterHunt2013}*{Thm. 3.3} to solutions of low regularity. 

\begin{proof}[Proof of Theorem \ref{thm: Uniqueness}] 
By Section \ref{sec: Notation}, we know that $\na \cdot \overline h \in CH_{loc}^{k-1}(M, T^*M, t)$. By Theorem \ref{thm: WellposednessLinearWaves}, we can define
\[
	V \in CH_{loc}^k(M, TM, t)
\]
as the unique solution to
\begin{align}
	\na^*\na V =& - \na \cdot \overline h^\sharp, \label{eq: DefiningV} \\
	V|_\S =& 0, \nonumber \\
	\na_\nu V|_\S =& \frac12 h(\nu, \nu) \nu + h(\nu, \cdot)^\sharp, \nonumber
\end{align}
where $\sharp: T^*M \to TM$ is the musical isomorphism with inverse $\flat:TM \to T^*M$. If $\supp(h) \subset J(K)$ for some subset $K \subset \S$, then \cite{BaerWafo2014}*{Rmk. 16} implies that $\supp(V) \subset J(K)$. By equation \eqref{eq: Killing_wave}, we have
\[
	\na \cdot \overline{\L_V g} = - \na^* \na V^\flat = \na \cdot \overline h,
\]
where $\overline{\L_V g} :=\L_V g - \frac12 \tr_g\left(\L_V g\right) g$. Hence
\begin{align*}
	0  =& 2 D\Ric(h - \L_Vg) \\
		=& \square_L (h - \L_Vg) + \L_{\na \cdot(\overline h- \overline{\L_Vg} )^\sharp } g \\
		=& \square_L (h - \L_Vg).
\end{align*}
Since $V \in CH^k_{loc}(M,TM, t)$, we know that $\L_V g \in CH^{k-1}_{loc}(M, S^2M, t)$, which implies that $h - \L_V g \in CH_{loc}^{k-1}(M, S^2M, t)$. Hence, if we knew that 
\begin{align}
	(h - \L_Vg)|_\S &= 0, \label{eq: ValueEq} \\
	\na_\nu(h - \L_Vg)|_\S &=0, \label{eq: NormalEq}
\end{align}
then Theorem \ref{thm: WellposednessLinearWaves} would imply that $h - \L_V g = 0$ as asserted. We start by showing \eqref{eq: ValueEq}. Since $V|_\S = 0$ and $\na_\nu V|_\S = \frac12 h(\nu, \nu) \nu + h(\nu, \cdot)^{\sharp}$ and $\h = 0$, we get for all $X, Y \in T\S$,
\begin{align*}
	h(X, Y) &= \h(X, Y) \\
		&= 0 \\
		&= g(\na_X V, Y) + g(\na_Y V, X) \\
		&= \L_Vg(X, Y), \\
	h(X, \nu) 
		&= g(\na_\nu V, X) \\
		&= g(\na_\nu V, X) + g(\na_X V, \nu) \\
		&= \L_Vg(\nu, X), \\
	h(\nu, \nu) &= 2g(\na_\nu V, \nu) \\
	&= \L_Vg(\nu, \nu).
\end{align*}
We continue by showing \eqref{eq: NormalEq}. Since $\m = 0$, we get for $X, Y \in T\S$ (recall Definition \ref{def: LinearisedFundamentalForms})
\[
    \na_\nu h(X,Y) = h(\nu, \nu) \tk(X,Y) + \na_X h(\nu, Y) + \na_Y h(\nu, X).
\]
Using $\h = 0$ and $V|_\S = 0$, we get
\begin{align*}
	\na_\nu \L_Vg(X,Y) &= g(\na^2_{\nu, X} V, Y) + g(\na^2_{\nu, Y} V, X) \\
		&= g(\na^2_{X, \nu} V, Y) + g(\na^2_{Y, \nu} V, X) + R(\nu, X, V, Y) + R(\nu, Y, V, X) \\
		&= \d_X g(\na_\nu V, Y) - g(\na_\nu V, \na_X Y) + \d_Y g(\na_\nu V, X) - g(\na_\nu V, \na_Y X) \\
		&= \d_X h(\nu, Y) - h(\nu, \n_XY) - \frac12 h(\nu, \nu) g(\nu, \n_X Y) \\
			& \qquad + \d_Y h(\nu, X) - h(\nu, \n_YX) - \frac12 h(\nu, \nu) g(\nu, \n_Y X) \\
			&= \n_Xh(\nu, Y) + \n_Y h(\nu, X) + h(\nu, \nu)\tk(X, Y) \\
			&= \na_\nu h(X, Y),
\end{align*}
since $\na_X\nu \in T\S$ and therefore $\na_{\na_X \nu}V = 0$. What remains to show is that $\na_\nu (h - \L_V g)|_\S(\nu, \cdot) = 0$. Recall that 
\[
 \na \cdot \overline{\L_V g} = \na \cdot \overline h,
\]
which is equivalent to
\begin{equation}\label{eq: TraceReversalUniqueness}
 \na \cdot {\L_V g}(W) - \frac12 \partial_W \tr_g(\L_V g) = \na \cdot h(W) - \frac12 \partial_W \tr_g(h),
\end{equation}
for all $W \in TM$. Note that from what is shown above, we know that $\tr_g(\L_V g)|_\S = \tr_g(h)|_\S$. Therefore, for  $X \in T\S$, we have $\partial_X \tr_g(\L_V g) = \partial_X \tr_g(h)$, so
\[
 \na \cdot {\L_V g}(X) = \na \cdot h(X),
\]
which simplifies to
\[
 \na_\nu {\L_V g}(X, \nu) = \na_\nu h(X, \nu).
\]
Instead inserting $\nu$ into equation \eqref{eq: TraceReversalUniqueness}, gives
\begin{align*}
 0 =& \na \cdot \overline{\L_V g}(\nu) - \na \cdot \overline h(\nu) \\
 =& \na \cdot ( \L_V g - h )(\nu) - \frac12 \partial_\nu \left( \tr_g(\L_V g) - \tr_g (h) \right) \\
 =& \na \cdot ( \L_V g - h )(\nu) - \frac12 \tr_g (\na_\nu \left( \L_V g - h \right)) \\
 =& - \na_\nu \left( \L_V g - h \right)(\nu, \nu) + \frac12 \na_\nu \left( \L_V g - h \right)(\nu, \nu) \\
 =& - \frac12 \na_\nu \left( \L_V g - h \right)(\nu, \nu).
\end{align*}
We conclude that
\[
 \na_\nu(h - \L_V g)(\nu, \nu) = 0.
\]
This shows that $h = \L_V g$. Lemma \ref{le: regularity_lie_derivative} implies the regularity of $V$.
\end{proof}

\subsection{Gauge producing initial data and gauge solutions} \label{sec: DegenerateInitialData} \label{sec: gpid}

In this section, we study the structure of the space of gauge solutions and gauge producing initial data. We consider from now on \emph{compactly supported initial data} and \emph{spatially compactly supported solutions}. The goal is to show that the spaces
\[
	\bigslant{\text{Initial data on }\S}{\text{Gauge producing initial data}}
\]
and
\[
	\bigslant{\text{Global solutions on }M}{\text{Gauge solutions}}
\]
equipped with the quotient topology are locally convex topological vector spaces. 

\begin{definition} \label{def: Solutions}
Define the \emph{solutions} of finite energy regularity $k \in \R \cup \{\infty\}$ as
\begin{equation*}
	 \Sol_{sc}^k(M, t) :=  CH^k_{sc}(M, S^2M, t) \cap \ker(D\Ric),
\end{equation*}
with the induced topology. 
\end{definition}

Since $D\Ric$ is a linear differential operator, it is continuous as an operator on distributions. 
Therefore, the solution space is a closed subspace and hence a locally convex topological vector space. 
Let us now define the subspace of gauge solutions. 

\begin{definition} \label{def: GaugeSolutions}
Define the \emph{gauge solutions} of finite energy and regularity $k \in \R \cup \{\infty\}$ as
\begin{align*}
	\G_{sc}^k(M,t) &:= \{ \L_V g \mid V \in CH^{k+1}_{sc}(M, TM, t) \} \subset \Sol_{sc}^k(M, t),
\end{align*}
with the induced topology. 
\end{definition}

We show later that the space of gauge solutions is a closed subspace of the solution space, which implies that the quotient space is a locally convex topological vector space. 
Let us define the space of solutions to the linearised constraint equation.

\begin{definition} \label{def: InitialData}
Define the \emph{initial data} of Sobolev regularity $k \in \R \cup \{\infty\}$ as
\begin{equation*}
	\ID_c^{k,k-1}(\S) := \left( H_c^{k}(\S, S^2\S) \times H_c^{k-1}(\S, S^2\S) \right)\cap \ker(D\Phi),
\end{equation*}
with the induced topology.
\end{definition}
Let
\begin{align*}
	\pi_\S: \Sol_{sc}^k(M, t) &\to \ID_c^{k,k-1}(\S)
\end{align*}
be the map that assigns to a solution the induced initial data, i.e. the linearised first and second fundamental forms. 
This map is given by Definition \ref{def: LinearisedFundamentalForms} and it is clear that $\pi_\S$ is continuous. 

\begin{definition} \label{def: GaugeProdInitialData}
Define the \emph{gauge producing initial data} of Sobolev regularity $k \in \R \cup \{\infty\}$ as
\begin{align*}
	\GP^{k,k-1}_c(\S) &:= \pi_\S(\G^k_{sc}(M,t)) \subset \ID_c^{k,k-1}(\S).
\end{align*}
\end{definition}

It will sometimes be necessary to consider only sections supported in a fixed compact set $K \subset \S$ or $J(K) \subset M$, for example $\ID_K^k(\S)$ or $\Sol_{J(K)}^k(M)$. 
The definitions in this case are analogous to Definitions \ref{def: Solutions}, \ref{def: GaugeSolutions}, \ref{def: InitialData} and \ref{def: GaugeProdInitialData}.

Let us study the space of gauge producing initial data $\GP^{k,k-1}_c(\S)$ in more detail. 
For $V \in CH^{k+1}_{sc}(M, TM, t)$, define $(N, \b) \in H_c^{k+1}(\S, \R \oplus T\S)$ by projecting $V|_\S$ to normal and tangential components, i.e. $V|_\S =: N\nu + \b$. 
Now define
\begin{align}
	\h_{N, \b} :=& 
		\L_{\b} \g +  2 \tk N, \label{eq: GaugeProducing1} \\
	\m_{N, \b} :=& \L_\b \tk + \Hess(N) + \left(2 \tk \circ \tk - \Ric(\tilde g) - (\tr_\g \tk)\tk \right) N. \label{eq: GaugeProducing2} 
\end{align}
We claim that $(\h_{N, \b}, \m_{N, \b}) = \pi_\S(\L_V g)$. 
Indeed, for each $X, Y \in T\S$, we have
 \begin{align*}
  \h_{N, \b}(X,Y) &= \L_V g(X,Y) \\
  	&= g(\n_X (\b + N\nu), Y) + g(\n_Y (\b + N\nu), X) \\
  	&= \L_\b \g(X,Y) + 2N\tk(X, Y), \\
  \m_{N, \b}(X,Y) &= - \frac12 \L_Vg(\nu, \nu) \tk(X,Y) - \frac12 \n_X \L_Vg(\nu, Y) \\
  	& \quad - \frac12 \n_Y \L_Vg(\nu, X) + \frac12 \n_\nu \L_V g(X,Y) \\ 
  	&= - g(\n_\nu V, \nu) \tk(X, Y) - \frac12 g(\n^2_{X,Y}V + \n^2_{Y,X}V, \nu) \\
  	&\quad + \frac12 R(\nu, X, V, Y) + \frac12 R(\nu, Y, V, X) \\
  	&= \L_\b \tk(X, Y) + \Hess(N)(X, Y) - \tilde \n_{\b} \tk(X, Y) \\
  	&\quad + \frac12 \left( \tilde \n_X \tk(Y, \b) +  \tilde \n_Y \tk(X, \b) \right) \\
  	&\quad + \frac12 R(\nu, X, \b, Y) + \frac12 R(\nu, Y, \b, X) \\
  	&\quad + N R(\nu, Y, \nu, X).
 \end{align*}
The classical Gauss and Codazzi equations now imply, using $\Ric(g) = 0$, that this coincides with \eqref{eq: GaugeProducing2}. 
In particular, $\GP^{k,k-1}_c(\S)$ can be defined intrinsically on $\S$ by equations \eqref{eq: GaugeProducing1} and \eqref{eq: GaugeProducing2} and is therefore independent of the chosen temporal function $t$ on $M$, as the notation suggests. 
We have shown the following lemma.
\begin{lemma}
For any $k \in \R \cup \{\infty\}$, the space of gauge producing initial data is given by
\begin{equation*} 
	\GP^{k,k-1}_c(\S) = \{(\h_{N, \b}, \m_{N, \b}) \text{ as in } \eqref{eq: GaugeProducing1} \text{ and } \eqref{eq: GaugeProducing2} \ | \ (N, \b) \in H_c^{k+1}(\S, \R \oplus T\S) \}.
\end{equation*}
\end{lemma}

We are now ready to prove that the space of gauge producing initial data is a closed subspace of the space of initial data.

\begin{lemma} \label{le: IDclosedness}
Let $k \in \R \cup \{\infty\}$. The space
\begin{equation*}
	\GP^{k,k-1}_c(\S) \subset \ID_{c}^{k, k-1}(\S), 
\end{equation*}
is a closed subspace. 
The statement still holds if we substitute $c$ with $K$, for a fixed compact subset $K \subset \S$. 
\end{lemma}

\begin{proof}
Consider the linear differential operator given by
\begin{align*}
	Q: H^{k+1}_c(\S, \R \oplus T\S) &\to H^{k}_c(\S, S^2\S) \times H^{k-1}_c(\S, S^2\S), \\
	(N, \b) &\mapsto (\h_{N, \b}, \m_{N, \b}).
\end{align*}
Since $\im(Q) = \GP^{k,k-1}_c(\S)$, the lemma is proven if we can show that $Q$ has closed image. 
We need to show that for each compact subset $K \subset \S$, $\im(Q) \cap H^k_K(\S) \times H^{k-1}_K(\S) \subset H^k_K(\S) \times H^{k-1}_K(\S)$ is closed. 
For a fixed compact subset $K \subset \S$, let us construct a set $L \subset \S$, containing $K$, such that if $\supp(Q(N, \b)) \subset L$ and $\supp(N, \b)$ is compact, then $\supp(N, \b) \subset L$. 
We construct $L$ as follows. Since $K$ is compact, $\partial K$ is compact, which implies that $M \backslash \mathring K$ has a finite amount of connected components.
Define $L$ to be the union of $K$ with all \emph{compact} connected components of $M \backslash \mathring K$. 
It follows that $L$ is compact, $K \subset L$ and that all components of $M \backslash \mathring L$ are non-compact. 
Let us show that $L$ has the desired properties. 
One calculates that the differential operator $P$, defined by
\begin{align*}
	P(N,\b) &:= \begin{pmatrix}
			-\tilde \na \cdot (\cdot) + \frac12 d \tr(\cdot)& 0 \\
			0 & -\tr(\cdot)
	\end{pmatrix}Q(N, \b) \\
	& = 
	\begin{pmatrix}
		 \tilde \na^* \tilde \na \b^\flat + l.o.t. \\
		 \tilde \na^* \tilde \na N + l.o.t.
	\end{pmatrix} \in H^{k-1}_K(\S, T^*\S \oplus \R)
\end{align*}
is a Laplace type operator. 
If $\supp(Q(N, \b)) \subset L$, and $\supp(N, \b)$ is compact, it follows that $\supp(P(N, \b)) \subset L$ and that $(M \backslash \mathring L) \cap \supp(N, \b) = \supp(N, \b) \backslash (\supp(N, \b) \cap \mathring L)$ is compact. Since each component of $M \backslash \mathring L$ was non-compact, Theorem \ref{thm: aronszajn} implies that $(N, \b) = 0$ on $M \backslash \mathring L$ and hence $\supp(N, \b) \subset L$ as claimed. Now if $Q(N_n, \b_n) \to (\h, \m)$ in $H^k_K(\S) \times H_K^{k-1}(\S)$, then $\supp(N_n, \b_n) \subset L$ and
\[
	P(N_n, \b_n) \to \begin{pmatrix}
			-\na \cdot (\h) + \frac12 d \tr(\h) \\
			-\tr(\m)
	\end{pmatrix}
\]
in $H^{k-1}_K(\S) \times H^{k-1}_K(\S)$. By Corollary \ref{cor: Laplace-type closed image}, we conclude that 
\[
	P: H^{k+1}_L(\S) \to H_L^{k-1}(\S)
\]
is an isomorphism onto its image and therefore there is a $(N, \b) \in H^{k+1}_L(\S)$ such that $(N_n, \b_n) \to (N, \b)$. We conclude that
\[
	(\h, \m) = \lim_{n \to \infty}Q(N_n, \b_n) = Q(N, \b),
\]
which finishes the proof.
\end{proof}

For later use, we need the following technical observation.

\begin{lemma} \label{le: SOLclosedness}
For $k \in \R \cup \{\infty\}$, 
\begin{equation*}
	\G_{sc}^k(M,t) = {\pi_\S}^{-1}(\GP^{k,k-1}_c(\S)).
\end{equation*}
In particular, 
\begin{equation*}
	\G_{sc}^k(M, t) \subset \Sol_{sc}^k(M, t)
\end{equation*}
is a closed subspace. The statement still holds if we substitute $c$ with $K$ and $sc$ with $J(K)$, for a fixed compact subset $K \subset \S$. 
\end{lemma}

\begin{proof}
Assume that $h \in \Sol_{sc}^k(M, t)$ and that 
\[
	\pi_\S(h) = \pi_\S(\L_Vg)
\]
for some $\L_V g \in \G_{sc}^k(M,t)$. Then $D\Ric(h - \L_Vg) = 0$ and $\pi_\S(h - \L_Vg) = 0$. By Theorem \ref{thm: Uniqueness}, there is a $W \in CH^{k+1}_{sc}(M,TM, t)$, such that
\[
	h = \L_V g + \L_W g = \L_{V+W}g
\]
which proves the statement. The smooth case is analogous. Since $\pi_\S$ is continuous, Lemma \ref{le: IDclosedness} implies the second statement. 
\end{proof}

The next lemmas give a natural way to understand the topology of the quotient spaces.

\begin{lemma} \label{le: GaugeInvariantInitialData}
	Let $k \in \R \cup \{\infty\}$. The l.c. topological vector space
	\begin{equation*}
	 	\bigslant{\ID_c^{k,k-1}(\S)}{\GP_c^{k, k-1}(\S)},
	\end{equation*}
	is the strict inductive limit of the l.c. topological vector spaces
	\begin{equation*}
		\bigslant{\ID_K^{k,k-1}(\S)}{\GP^{k,k-1}_{K}(\S)},
	\end{equation*}
	for compact subsets $K \subset \S$, with respect to the natural inclusions. In particular, it is an LF-space.
\end{lemma}
\begin{proof}
Let us simplify notation by writing $\ID_K := \ID_K^{k,k-1}(\S)$, $\ID_c := \ID_c^{k,k-1}(\S)$, $\GP_K := \GP^{k,k-1}_{K}(\S)$ and $\GP_{c} := \GP^{k,k-1}_{c}(\S)$. 
Let $K_1 \subset K_2 \subset \hdots \bigcup_{n \in \N} K_n = \S$ be an exhaustion by compact subsets. 
Since $\ID_{K_n} \subset \ID_{K_{n+1}}$ is closed for all $n \in \N$, the strict inductive limit exists. 
By (a slight modification of) Lemma \ref{le: two topologies} we conclude that $\ID_c$ is the strict inductive limit space of the spaces $\ID_K$. 
Similarly, $\bigslant{\ID_{K_n}}{\GP_{K_n}} \subset \bigslant{\ID_{K_{n+1}}}{\GP_{K_{n+1}}}$ are closed and hence the strict inductive limit topology on $\bigslant{\ID_c}{\GP_c(\S)}$ exists. 
Call the quotient topology $\tau_{quot}$ and the strict inductive limit topology $\tau_{ind}$. The map
\[
	\left(\bigslant{\ID_c}{\GP_c}, \tau_{ind}\right) \to \left( \bigslant{\ID_c}{\GP_c}, \tau_{quot}\right)
\]
is continuous if and only if the restriction
\[
	\bigslant{\ID_{K_n}}{\GP_{K_n}} \to \left(\bigslant{\ID_c}{\GP_c}, \tau_{quot} \right)
\]
is continuous for all $n \in \N$. 
But this is clear, since $\ID_{K_n} \to \ID_c$ is continuous. 
Conversely, the map 
\[
	\left( \bigslant{\ID_c}{\GP_c}, \tau_{quot}\right) \to \left(\bigslant{\ID_c}{\GP_c}, \tau_{ind}\right)
\]
is continuous if and only if
\[
	\ID_c \to \left(\bigslant{\ID_c}{\GP_c}, \tau_{ind}\right)
\]
is continuous. 
This is however equivalent to 
\[
	\ID_{K_n} \to \left(\bigslant{\ID_c}{\GP_c}, \tau_{ind}\right)
\]
being continuous for all $n \in \N$. 
This true if and only if 
\[
	\bigslant{\ID_{K_n}}{\GP_{K_n}} \to \left(\bigslant{\ID_c}{\GP_c}, \tau_{ind}\right)
\]
is continuous for all $n \in \N$. 
But this is clear, by construction of the strict inductive limit topology. 
This establishes the claimed homeomorphism.
Since the quotient spaces $\bigslant{\ID_{K_n}}{\GP_{K_n}}$ are Fréchet spaces, its strict inductive limit will be a LF-space by definition.
\end{proof}

\begin{lemma} \label{le: GaugeInvariantSolutions}
	Let $k \in \R \cup \{\infty\}$. The l.c. topological vector space
	\begin{equation*}
		\bigslant{\Sol_{sc}^k(M, t)}{\G_{sc}^k(M,t)},
	\end{equation*}
	is the strict inductive limit of the l.c. topological vector spaces
	\begin{equation*}
		\bigslant{\Sol_{J(K)}^k(M, t)}{\G_{J(K)}^k(M,t)},
	\end{equation*}
	for compact subsets $K \subset \S$, with respect to the natural inclusions. In particular, it is an LF-space.
\end{lemma}
\begin{proof}
The proof is analogous to the proof of Lemma \ref{le: GaugeInvariantInitialData}, using Lemma \ref{le: two topologies2} instead of Lemma \ref{le: two topologies}. 
\end{proof}

\subsection{Continuous dependence on initial data} \label{sec: Main}

Let us now state and prove the main result of this section, the well-posedness of the Cauchy problem of the linearised Einstein equation. Recall Section \ref{sec: DegenerateInitialData} for the definitions of the function spaces below.

\begin{thm}[Well-posedness of the Cauchy problem] \label{thm: Wellposedness} 
Let $k \in \R \cup \{\infty\}$. The linear solution map
\begin{equation*}
	\mathrm{Solve}^k: \bigslant{\ID_c^{k,k-1}(\S)}{\GP^{k,k-1}_{c}(\S)} \to \bigslant{\Sol_{sc}^k(M, t)}{\G_{sc}^k(M,t)}
\end{equation*}
is an isomorphism of locally convex topological vector spaces. In fact, both spaces are $LF$-spaces.
\end{thm}

The theorem implies that the equivalence class of solutions depends continuously on the equivalence class of initial data. Since projection maps are continuous and surjective, we immediately get the following corollary.

\begin{cor}[Continuous dependence on initial data] \label{cor: ContDependence}
Let $k \in \R \cup \{\infty\}$. The linear solution map
\begin{equation*}
	\mathrm{\widetilde{Solve}^k}: \ID_c^{k,k-1}(\S) \to \bigslant{\Sol_{sc}^k(M, t)}{\G_{sc}^k(M,t)}
\end{equation*}
is continuous and surjective. 
\end{cor}

Before proving the theorem, let us discuss some more remarks and corollaries.

\begin{remark}[Distributional initial data]
Since any compactly supported distribution is of some real Sobolev regularity, any compactly supported distributional section lies in some $H^k_c(\S)$. Therefore Theorem \ref{thm: Wellposedness} covers the case of any compactly supported distributional initial data.  
\end{remark}

A priori, the solution spaces depend on the time function. After quoting out the gauge solutions, this is not the case anymore.

\begin{cor}[Independence of the Cauchy temporal function] \label{cor: IndependenceOfTime}
Let $t$ and $\tau$ be Cauchy temporal functions on $M$. Then for every $k \in \R \cup \{\infty\}$ there is an isomorphism
\[
	\bigslant{\Sol_{sc}^k(M, t)}{\G_{sc}^k(M,t)} \to \bigslant{\Sol_{sc}^k(M, \tau )}{\G^k_{sc}(M,\tau)}
\]
which is the identity map on smooth solutions.
\end{cor}

\begin{proof}
The proof is analogous to the proof of \cite{BaerWafo2014}*{Cor. 18}, using Theorem \ref{thm: Wellposedness}.
\end{proof}

As a final observation, let us note that if $\S$ is compact, we obtain a natural Hilbert space structure on the solution space.

\begin{cor}[Hilbert space structure on the solution space]
Let $k \in \R$. In case $\S$ is compact, Theorem \ref{thm: Wellposedness} implies that 
\[
	\bigslant{\Sol^k(M, t)}{\G^k(M,t)}
\]
carries a Hilbert space structure, induced by $\mathrm{Solve}^k$. In case $k = \infty$, it is a Fréchet space.
\end{cor}

\begin{proof}
Since $\S$ is compact, $\bigslant{\ID^{k,k-1}(\S)}{\GP^{k,k-1}(\S)}$ carries a Hilbert space induced from the Sobolev space in case $k < \infty$. 
In case $k = \infty$, it carries a natural Fréchet space induced by the Fréchet space structure on the smooth sections. 
Theorem \ref{thm: Wellposedness} and Corollary \ref{cor: IndependenceOfTime} imply that we get an induced Hilbert space structure on $\bigslant{\Sol^k(M, t)}{\G^k(M,t)}$ independent of the choice of Cauchy hypersurface $\S$.
\end{proof}

Let us turn to the proof of Theorem \ref{thm: Wellposedness}.

\begin{lemma} \label{le: FixedCompact}
Let $k \in \R \cup \{\infty\}$ and fix a compact subset $K \subset \S$. The linear map
\begin{equation*}
	\mathrm{Solve}_K^k: \bigslant{\ID_K^{k,k-1}(\S)}{\GP_K^{k,k-1}(\S)} \to \bigslant{\Sol_{J(K)}^k(M, t)}{\G_{J(K)}^k(M,t)}
\end{equation*}
is an isomorphism of topological vector spaces.
	
\end{lemma}

\begin{proof} Lemma \ref{le: IDclosedness} and Lemma \ref{le: SOLclosedness} imply that the quotient spaces are well defined Fréchet spaces. By Theorem \ref{thm: Existence}, Theorem \ref{thm: Uniqueness} and Lemma \ref{le: SOLclosedness}, the map $\mathrm{Solve}^k_K$ is a well defined linear bijection. We prove that it indeed is an isomorphism of topological vector spaces. Recall that the map that assigns to each solution its initial data 
\[
	\pi_\S: \Sol^k_{J(K)}(M, t) \to \ID^{k,k-1}_K(\S)
\]
is continuous. By definition of the quotient space topology, $\pi_\S$ induces a continuous map 
\[
	\hat \pi_\S: \bigslant{\Sol^k_{J(K)}(M, t)}{\G^k_{J(K)}(M,t)} \to \bigslant{\ID^{k,k-1}_K(\S)}{\GP^{k,k-1}_{K}(\S)}
\]
between Fréchet spaces. Since $\hat \pi_\S$ is the inverse of $\mathrm{Solve}^k_K$, the open mapping theorem for Fréchet spaces implies the statement. 
\end{proof}

\begin{proof}[Proof of Theorem \ref{thm: Wellposedness}]
Again, Lemma \ref{le: IDclosedness} and Lemma \ref{le: SOLclosedness} imply that the quotient spaces are well defined topological vector spaces. By Theorem \ref{thm: Existence}, Theorem \ref{thm: Uniqueness} and Lemma \ref{le: SOLclosedness}, the map $\mathrm{Solve}^k$ is a well defined linear bijection. Therefore it remains to prove that it is an isomorphism of topological vector spaces. By Lemma \ref{le: FixedCompact}, the map
\begin{align*}
	\bigslant{\ID_K^{k,k-1}(\S)}{\GP_{K}^{k,k-1}(\S)} & \stackrel{\text{Solve}_K^k}{\longrightarrow} \bigslant{\Sol_{J(K)}^k(M, t)}{\G_{J(K)}^k(M,t)} \\ 
	& \hookrightarrow \bigslant{\Sol_{sc}^k(M, t)}{\G_{sc}^k(M,t)}
\end{align*}
is continuous for every compact subset $K \subset \S$. By Lemma \ref{le: GaugeInvariantInitialData}, this implies that $\text{Solve}^k$ is continuous. Similarly, by Lemma \ref{le: FixedCompact}, the composed map
\begin{align*}
	\bigslant{\Sol^k_{J(K)}(M, t)}{\G^k_{J(K)}(M,t)}&\stackrel{\hat \pi_\S}{\longrightarrow} \bigslant{\ID^{k,k-1}_K(\S)}{\GP^{k,k-1}_{K}(\S)} \\
	& \hookrightarrow \bigslant{\ID^{k,k-1}_c(\S)}{\GP^{k,k-1}_{c}(\S)}
\end{align*}
is continuous for every compact subset $K \subset \S$. By Lemma \ref{le: GaugeInvariantSolutions}, this implies that $\left(\text{Solve}^k\right)^{-1}$ is continuous. 
\end{proof}

\section{The linearised constraint equations} \label{ch: lin_constraint}

In order to apply Theorem \ref{thm: Wellposedness} in practice, it is necessary to understand the space
\[
	\bigslant{\ID_c^{k,k-1}(\S)}{\GP^{k,k-1}_{c}(\S)}.
\]
In this section, we show that this space can be quite well understood if $\S$ is compact, $\Scal(\g) = 0$ and $\tk = 0$. 
The idea is inspired by the following classical result:
If $\Ric(\g) = 0$ and $\S$ is compact, then equivalence classes of initial data are essentially in one-to-one correspondence with the divergence- and trace free tensors on $\S$ (\enquote{transverse traceless tensors} or \enquote{TT-tensors}). 
The advantage of this observation comes from the following well-known fact.
For any $(0,2)$-tensor $\a$ on $\S$, there is a unique decomposition
\begin{equation} \label{eq: decomposition}
	\a = \h + L\o + \phi \g,
\end{equation}
where $\h$ is a $TT$-tensor, $\o$ is a one-form, $L$ is the conformal Killing operator and $\phi$ is a function. 
Now, the problem is that if $\Ric(\g) \neq 0$, then TT-tensors \emph{do not} solve the linearised constraint equation in general. 
The goal of this section is to generalise the decomposition \eqref{eq: decomposition} to the case when $\Scal(\g) = 0$.
Let us therefore assume in this section that $\S$ is compact, $\Scal(\g) = 0$ and $\tk = 0$, which is obviously a solution of the non-linear constraint equations (\ref{eq: ham_constraint} - \ref{eq: momentum_constraint}). 

As mentioned in the introduction, it turns out that equation (\ref{eq: Eq1firstFF} - \ref{eq: Eq2secondFF}) will be relevant for this problem. For any $k \in \R \cup \{\infty \}$, let $\Gamma^k(\S) \subset H^k(\S, S^2\S)$ denote the $H^k$-solutions to the equations \eqref{eq: Eq1firstFF} and \eqref{eq: Eq2firstFF} and let $\Gamma^{k-1}(\S) \subset H^{k-1}(\S, S^2\S)$ denote the $H^{k-1}$-solutions to \eqref{eq: Eq1secondFF} and \eqref{eq: Eq2secondFF}.
The following proposition is a special case of Moncrief's classical splitting theorem \cite{Moncrief1975}, generalised to any Sobolev degree. It can be seen as a \enquote{gauge choice} for the initial data. 

\begin{prop} \label{prop: Moncrief_splitting}
Let $k \in \R \cup \{\infty\}$. Assume that $(\S, \g)$ is a closed manifold with vanishing scalar curvature and that $\tk = 0$. For any $k \in \R \cup \{\infty\}$, the map
\begin{align*}
	\Gamma^{k}(\S) \times \Gamma^{k-1}(\S) &\to \bigslant{\ID^{k,k-1}(\S)}{\GP^{k,k-1}(\S)}, \\
	(\h, \m) &\mapsto [(\h, \m)],
\end{align*}
is an isomorphism of Banach spaces.
\end{prop}

Since the proof for arbitrary Sobolev degree is hard to find, we give a simple proof of this proposition later. By Theorem \ref{thm: Wellposedness}, we conclude that the composed map
\begin{align*}
	\Gamma^{k}(\S) \times \Gamma^{k-1}(\S) \to \bigslant{\ID^{k,k-1}(\S)}{\GP^{k,k-1}(\S)} \stackrel{Solve^k}{\to} \bigslant{\Sol^k(M, t)}{\G^k(M,t)},
\end{align*}
is an isomorphism of Banach spaces. 

Let us now state the main result of this section. Let 
\[
	L\o := \L_{\o^\sharp} \g - \frac{2}{\dim(\S)} (\tilde \n \cdot \o) \g
\]
denote the conformal Killing operator on one-forms.
\begin{thm}\label{thm: initial_data_split}
Assume that $(\S, \g)$ is a closed Riemannian manifold of dimension $n \geq 2$  with $\Scal(\g) = 0$ and $\tk = 0$. Let $k \in \R \cup \{\infty\}$. Then for each $(\a, \b) \in H^{k}(\S, S^2\S) \times H^{k-1}(\S, S^2\S)$, there is a unique decomposition
\begin{align*}
	\a &= \h + L \o + C \Ric(\g) + \phi \g, \\
	\b &= \m + L \eta + C' \Ric(\g) + \psi \g,
\end{align*}
where $(\h, \m) \in \Gamma^{k}(\S) \times \Gamma^{k-1}(\S)$, $(\o, \eta) \in H^{k+1}(\S, T^*\S) \times H^{k}(\S, T^*\S)$, $(C, C') \in \R^2$ and $(\phi, \psi) \in H^{k}(\S, \R) \times H^{k-1}(\S, \R)$ such that $\phi[1] = 0 = \psi[1]$.
\end{thm}

Here $\phi[1]$ means the distribution $\phi$ evaluated on the test function with value $1$. If $\phi \in L^1_{loc}$ then $\phi[1] = \int_\S \phi d\mu_{\g}$.

\begin{example}
Let us give two examples of closed Riemannian manifolds with vanishing scalar curvature. 
\begin{itemize}
\item For each $n \in \N$, the flat torus $T^n := \R^n / \Z^n$ is flat, in particular $\Scal(\g) = 0$.
\item For each $m \in \N$, there is a Berger metric on $S^{4m-1}$ with vanishing scalar curvature. In case $m = 1$, the scalar flat Berger metric is given by $\frac52 {\sigma_1}^2 + {\sigma_2}^2 + {\sigma_3}^2$, where $\sigma_1, \sigma_2, \sigma_3$ are orthonormal left invariant one-forms on $S^3$. Note that this metric does not have vanishing Ricci curvature.
\end{itemize}
On these manifolds, Theorem \ref{thm: initial_data_split} applies.
\end{example}

\begin{remark}
Note that Theorem \ref{thm: initial_data_split} is equivalent to showing that
\begin{align*}
	H^{k}(\S, S^2\S) &= \Gamma_1^{k}(\S) \oplus \im(L) \oplus \R \Ric(\g) \oplus \widehat H^k(\S, \R) \g, \\
	H^{k}(\S, S^2\S) &= \Gamma_2^{k}(\S) \oplus \im(L) \oplus \R \Ric(\g) \oplus \widehat H^k(\S, \R) \g, 
\end{align*}
where $L: H^{k+1}(\S, T^*\S) \to H^{k}(\S, S^2\S)$ and $\widehat H^k(\S, \R) := \{ \phi \in H^k(\S, \R) \mid \phi[1] = 0 \}$. 
\end{remark}

\begin{remark}[TT-tensors]
Note that if $\Ric(\g) = 0$, equations (\ref{eq: Eq1firstFF} - \ref{eq: Eq2secondFF}) imply that $\h$ and $\m$ are TT-tensors or a constant multiple of the metric. Moreover, if $\Ric(\g) = 0$, then Theorem \ref{thm: initial_data_split} simplifies essentially to the classical split mentioned in the beginning of this section.
\end{remark}

Before proving Proposition \ref{prop: Moncrief_splitting} and our main result Theorem \ref{thm: initial_data_split}, let us use Proposition \ref{prop: Moncrief_splitting} to show that there are arbitrarily irregular non-gauge gravitational waves. 

\begin{example}[Arbitrarily irregular non-gauge solutions] \label{ex: arbitrarily irregular}
Consider the flat torus $((S^1)^3, \g)$ with coordinates $(x^1, x^2, x^3)$. Let $\delta^{(n)}$ denote the $n$-th derivative of the Dirac distribution on $S^1$ with support at some fixed point in $S^1$. The tensor defined by
\[
	\h(x^1, x^2, x^3) := \delta^{(n)}(x^3)dx^1 \otimes dx^2
\]
is a TT-tensor. Moreover, $\h \in H^{-n-1}(\S) \backslash H^{-n}(\S)$. Combining Proposition \ref{prop: Moncrief_splitting} with Theorem \ref{thm: Wellposedness}, this shows that there are arbitrarily irregular non-gauge gravitational waves on the spatially compact Minkowski spacetime $M = \R \times (S^1)^3$.
\end{example}

We start by giving a simple proof of Proposition \ref{prop: Moncrief_splitting}. The proof is more elementary than the original one by Moncrief, since we only consider the case when $\tk = 0$.

\begin{proof}[Proof of Proposition \ref{prop: Moncrief_splitting}]
Note first that $(\h, \m) \in \ID^{k,k-1}(\S)$ if and only if
\begin{align*}
	\tilde \na \cdot (\tilde \na \cdot \h - d \tr_\g \h) - \g(\Ric(\g), \h) &= 0, \\
	\tilde \na \cdot (\m - (\tr_\g\m)\g) &= 0.
\end{align*}
The gauge producing initial data $\GP^{k,k-1}(\S)$ is in this case given by the image of
\begin{align*}
	P: H^{k+1}(\S, T\S \oplus \R) &\to H^k(\S, S^2\S) \times H^{k-1}(\S, S^2\S), \\
	(\b, N) &\mapsto (\L_\b \g, \Hess(N) - \Ric(\g) N).
\end{align*}
The formal adjoint of $P$ is given by
\[
	P^*(\h, \m) = (- 2 \tilde \n \cdot \h, \tilde \na \cdot \tilde \na \cdot \m - \g(\Ric(\g), \m)).
\]
Recall by Lemma \ref{le: IDclosedness}, that we know that $\im(P) = \GP^{k,k-1}(\S) \subset \ID^{k,k-1}(\S)$ is closed. We claim that 
\begin{equation} \label{eq: part_fredholm}
	H^k(\S, S^2\S) \oplus H^{k-1}(\S, S^2\S) = \im(P) \oplus \ker(P^*).
\end{equation}
We first prove this when $k \leq 0$. Define 
\begin{align*}
	P_0: H^{1} (\S, T\S) \times H^{2}(\S, \R) &\to L^2(\S, S^2\S \oplus S^2\S), \\
	(\b, N) &\mapsto P(\b, N).
\end{align*}
It follows that
\begin{align*}
	L^2(\S, S^2\S \oplus S^2\S) &= \overline{\im(P_0)} \oplus \ker(P_0^*) \\
		&\subset \im(P) \oplus \ker(P^*) \\
		&\subset H^{k}(\S, S^2\S) \oplus H^{k-1}(\S, S^2\S).
\end{align*}
Since $\im(P) \oplus \ker(P^*) \subset H^{k}(\S, S^2\S) \oplus H^{k-1}(\S, S^2\S)$ is closed and $L^2(\S, S^2\S \oplus S^2\S) \subset H^{k}(\S, S^2\S) \oplus H^{k-1}(\S, S^2\S)$ is dense, we have proven \eqref{eq: part_fredholm} when $k \leq 0$. Assume now that $k > 0$ and that $(\h, \m) \in H^{k}(\S) \times H^{k-1}(\S)$. Since we know equation \eqref{eq: part_fredholm} when $k = 0$, we conclude that there is $(N, \b) \in H^{1}(\S)$ and $(\h_0, \m_0) \in L^2(\S) \times H^{-1}(\S)$ such that $P^*(\h_0, \m_0) = 0$ and
\[
	(\h, \m) = P(N, \b) + (\h_0, \m_0).
\]
It follows that $P^*P(N, \b) = P^*(\h, \m) \in H^{k-1}(\S) \times H^{k-3}(\S)$. Note that
\[
	\begin{pmatrix}
	\tilde \n^* \tilde \n & 0 \\
	0 & 1
	\end{pmatrix} \circ P^*P: H^{k+1}(\S) \to H^{k-3}(\S)
\]
is an elliptic differential operator. It follows that $(N, \b) \in H^{k+1}(\S)$ and hence $(\h_0, \m_0) = (\h, \m) - P(N, \b) \in H^{k}(\S) \times H^{k-1}(\S)$. This proves the claim for $k > 0$.

Since $\im(P) = \GP^{k,k-1}(\S) \subset \ID^{k,k-1}(\S)$, it follows now that
\[
	\ID^{k,k-1}(\S) = \GP^{k,k-1}(\S) \oplus \left( \ID^{k,k-1}(\S) \cap \ker(P^*) \right).
\]
One checks that $\Gamma^{k}(\S) \times \Gamma^{k-1}(\S) = \ID^{k,k-1}(\S) \cap \ker(P^*)$.  This concludes the proof.
\end{proof}

Let us turn to the proof of Theorem \ref{thm: initial_data_split}. Note that $\h = \a - L \o - C \Ric(\g) - \phi \g \in \Gamma_1^k(\S)$ if and only if
\begin{align*}
	\Delta \phi - \frac1n \g(\Ric(\g), L\o) &= -\frac1n \g(\Ric(\g), \a) + \frac1n \Delta \tr_\g \a + \frac C n \g(\Ric(\g), \Ric(\g)), \\
	L^*L\o - 2 d\phi &= - 2 \tilde \n \cdot \a
\end{align*}
and $\m = \b - L \eta - C' \Ric(\g) - \psi \g \in \Gamma_2^{k-1}(\S)$ if and only if
\begin{align*}
	\Delta \psi + \frac1n \g(\Ric(\g), L\eta) &= \frac1n \g(\Ric(\g), \b) + \frac1n \Delta \tr_\g \b - \frac {C'} n \g(\Ric(\g), \Ric(\g)), \\
	L^*L\eta + 2(n-1)d\psi &= - 2\tilde \na \cdot (\b - (\tr_\g \b)\g),
\end{align*}
using that $\tilde \n \cdot \Ric(\g) = \frac12 d \Scal(\g) = 0$ and $\tr_\g \Ric(\g) = \Scal(\g) = 0$. The idea is to consider the right hand side as given and find $(\phi, \o)$ and $(\psi, \eta)$ solving the equations. 
The idea is to consider the left hand side as an elliptic operator and calculate its kernel and cokernel.

Let $\L \o := \L_{\o^\sharp}\g$ denote the Killing operator on one-forms $\o$.

\begin{lemma} \label{le: almost bijectivity of P}
Assume that $(\S, \g)$ is a closed Riemannian manifold of dimension $n \geq 2$ such that $\Scal(\g) = 0$. Let $a, b \in \R$ such that $0 < ab < 2$. For any $k \in \R \cup \{\infty\}$, consider the elliptic differential operator
\begin{align*}
	&P: H^{k+2}(\S, \R \oplus T^*\S) \to H^k(\S, \R \oplus T^*\S), \\
	&P(\phi, \o) := \begin{pmatrix}
						\Delta \phi + a \g(\Ric(\g), L\o)) \\
						L^*L \o + b d\phi
					\end{pmatrix}.
\end{align*}
Then
\begin{align*}
	\ker(P) = \ker(P^*) = \ker(d) \oplus \ker(\L),
\end{align*}
i.e. both kernels consist only of the constant functions and Killing one-forms.
\end{lemma}

In our case, we have that $(a, b) = (-\frac 1n, -2)$, which implies that $ab = \frac{2}{n}$ and secondly that $(a, b) = (\frac1n, 2(n-1))$, which implies that $ab = \frac{2(n-1)}{n}$. 
In both cases $0 < ab < 2$, for all $n \geq 2$, so the lemma applies.

We will use the following differential operators acting on one-forms $\o$ and functions $\phi$ on $\S$:
\begin{align*}
	\delta \o :=& - \na \cdot \o, \\
	\Delta \o :=& (d\delta + \delta d) \omega, \\
	\Delta \phi :=& (d\delta + \delta d)\phi = \delta d \phi.
\end{align*}

\begin{proof}[Proof of Lemma \ref{le: almost bijectivity of P}]
Let us first note that
\begin{equation} \label{eq: L*L}
	L^*L \o = 2 \Delta \o - 4 \Ric(\g)(\o^\sharp) + \left( 2 - \frac4n \right)d\delta \o,
\end{equation}
for any one form $\o$.
We start by showing that $\ker(P) = \ker(d) \oplus \ker(\L)$. For this, assume that 
\[
	P(\phi, \o) = 0.
\]
It follows that $\delta(L^*L \o) = - b \Delta \phi$. On the other hand, using $\tilde \n \cdot \Ric_{\g} = \frac12 d\Scal(\g) = 0$ it follows that
\begin{align*}
	\delta(L^*L \o) &= \left(4-\frac4n \right) \Delta \delta \o + 2 \g(\Ric(\g), L\o) \\
		&= \left(4-\frac4n \right) \Delta \delta \o - \frac{2}{a} \Delta \phi,
\end{align*}
where in the last line we have used that $P(\phi, \o) = 0$. Combining these two results gives
\[
	\Delta \left(\left(4- \frac4n \right)\delta \o +\left(b - \frac2a \right)\phi\right) = 0.
\]
Since $\S$ is closed, all harmonic functions are constant and hence
\[
	\phi = \frac{4 - \frac4n}{\frac2a - b}\delta \o + C,
\]
where $C$ is constant, which implies that
\[
	L^*L \o = - b d \phi = \frac{4 - \frac4n}{1 - \frac2{ab}} d\delta \o.
\]
Since $0 < ab < 2$, it follows that 
\[
	\norm{L\o}_{L^2}^2 = \frac{4 - \frac4n}{1 - \frac2{ab}} \norm{\delta \o}_{L^2}^2 \leq 0.
\]
We conclude that $L\o = 0$ and $\delta \o = 0$. Hence $\o$ is a Killing one-form. It follows that $d \phi = 0$ as claimed.

We continue by calculating $\ker (P^*)$. From equation \eqref{eq: L*L}, we get
\[
	- 2a\Ric(\g)(\grad(\phi), \cdot) = -2a \left( 1 - \frac1 n \right) d \Delta \phi + \frac a 2 L^*L d\phi.
\]
Assuming that $P^*(\phi, \o) = 0$, it follows that
\[
	L^*L\left(\o + \frac a2 d\phi \right) = 2a\left( 1 - \frac 1n \right)d\Delta \phi.
\]
Again using $P^*(\phi, \o) = 0$ we conclude that
\begin{align*}
	\norm{L\left(\o + \frac a2 d\phi \right)}^2 &= 2a\left( 1 - \frac 1n \right) \langle d \Delta \phi, \o + \frac a 2 d\phi \rangle \\
	&= a^2 \left( 1 - \frac 2 {ab} \right)\left( 1 - \frac 1n \right) \norm{\Delta \phi}^2 \\
	&\leq 0,
\end{align*}
since $0 < ab < 2$, which implies that
\[
	L\left(\o + \frac a2 d\phi \right) = 0
\]
and hence $\Delta \phi = 0$. Since $\S$ is closed, it follows that $\phi$ is constant and hence $L\o = 0$. Since $b \neq 0$, it follows that $\delta \o = 0$ and hence $\o$ is a Killing one-form as claimed.
\end{proof}

\begin{proof}[Proof of Theorem \ref{thm: initial_data_split}]
We first show that $\widehat H^k(\S, \R) \g \oplus \im(L) \oplus \R \Ric(\g)$ really is a direct sum. 
Since $\Scal(\g) = 0$, we have for all $f \in \widehat H^k(\S, \R)$ that
\[
	f\g[C \Ric(\g)] = f[C\g(\g, \Ric(\g))] = f[0] = 0
\]
and hence $\widehat H^k(\S, \R) \g \cap \R \Ric(\g) = \{0\}$. 
Since $\o \mapsto L\o$ has injective principal symbol, Lemma \ref{le: Fredholm alternative closed mfld} implies that
\[
	H^{k}(\S, S^2\S) = \im(L) \oplus \ker(L^*).
\]
Since $\Scal(\g) = 0$, $L^*(\Ric(\g)) = -2\tilde \na \cdot \Ric(\g) = - d\Scal(\g) = 0$ and hence 
\[
	\R\Ric(\g) \subset \ker(L^*),
\]
which implies that $\R \Ric(\g) \cap \im(L) = \{0\}$. 
That $\widehat H^k(\S, \R)\g \cap \im(L) = \{0\}$ is clear, since $\tr_g(L\o) = 0$. 
This proves the first claim. Let us now prove that $\left( \widehat H^k(\S, \R) \g \oplus \im(L) \oplus \R \Ric(\g) \right) \cap \Gamma_1^k(\S) = \{0\}$. 
For this, assume that 
\[
	0 = \h + \phi \g + L\o + C \Ric(\g) \in \Gamma_1^k(\S),
\]
with $\phi \in \widehat H^k(\S, \R)$ and $\o \in H^{k+1}(\S, T^*\S)$. We know that $\h \in \Gamma_1^k(\S)$ if and only if
\[
	P(\phi, \o) = \begin{pmatrix}-\frac{C}{n} \g(\Ric(\g), \Ric(\g)) \\ 0 \end{pmatrix},
\]
with $(a, b) = (-\frac{1}{n}, -2)$. 
By Lemma \ref{le: Fredholm alternative closed mfld} and Lemma \ref{le: almost bijectivity of P}, it follows that $C\g(\Ric(\g), \Ric(\g))$ must be orthogonal to the constant functions, i.e. 
\[
	\int_{\S}C\g(\Ric(\g), \Ric(\g)) d\mu_\g= 0.
\]
Since $\g(\Ric(\g), \Ric(\g)) \geq 0$, we conclude that either $\Ric(\g) = 0$ or $C = 0$ which in both cases implies $C \Ric(\g) = 0$. 
Hence $(\phi, \o) \in \ker(P)$, which by Lemma \ref{le: almost bijectivity of P} implies that $\phi$ is constant and $\o$ is a Killing one-form. 
Hence $L \o = 0$ and since $0 = \phi[1] = \int_\S \phi d \mu_\g$, it follows that $\phi = 0$. 
This proves $\left( \widehat H^k(\S, \R) \g \oplus \im(L) \oplus \R \Ric(\g) \right) \cap \Gamma_1^k(\S) = \{0\}$. 
Similarly, one proves $\left( \widehat H^k(\S, \R) \g \oplus \im(L) \oplus \R \Ric(\g) \right) \cap \Gamma_2^k(\S) = \{0\}$.

It remains to show that
\[
	H^k(\S, S^2\S) \subseteq \widehat H^k(\S, \R) \g \oplus \im(L) \oplus \R \Ric(\g) \oplus \Gamma_1^k(\S).
\]
Given $\a \in H^k(\S, S^2\S)$ we want to find $\phi \in \widehat{H}^k(\S, \R)$ and $\o \in H^{k+1}(\S, T^*\S)$ such that $\h := \a - \phi \g - L\o - C \Ric \in \Gamma_1^k(\S)$. Note that $\h \in \Gamma_1^k(\S)$ if and only if
\begin{equation} \label{eq: defining_phi_omega}
	P(\phi, \o) = \begin{pmatrix} - \frac1n \g(\a, \Ric(\g)) + \frac1n \Delta \tr_\g \a +  \frac{C}{n}\g(\Ric(\g), \Ric(\g)) \\ - 2\tilde \na \cdot \a \end{pmatrix}.
\end{equation}
By Lemma \ref{le: Fredholm alternative closed mfld} and Lemma \ref{le: almost bijectivity of P} we find $(\phi, \o) \in H^{k}(\S, \R \oplus T^*\S)$ if and only if we choose 
\[
	C:= \frac{\g(\a, \Ric(\g))[1]}{\int_\S g(\Ric, \Ric) d \mu_\g},
\]
when $\Ric(\g) \neq 0$. 
If $\Ric(\g) = 0$, it does not matter how we choose $C$, $C\Ric(\g) = 0$ anyway. 
What remains is to show that $L\o \in H^{k}(\S, S^2\S)$, up to now we only know that $L\o \in H^{k-1}(\S, S^2\S)$. 
But from equation \eqref{eq: defining_phi_omega}, we know that $L^*L \o = 2 d\varphi - 2 \tilde \na \cdot \a \in H^{k-1}(\S, T^*\S)$. 
Elliptic regularity theory implies that in fact $\o \in H^{k+1}(\S, T^*\S)$ which implies that $L\o \in H^{k}(\S, S^2\S)$. 
The inclusion $H^k(\S, S^2\S) \subseteq \widehat H^k(\S, \R) \g \oplus \im(L) \oplus \R \Ric(\g) \oplus \Gamma_2^k(\S)$ is proven analogously.
\end{proof}

\appendix

\section{Some linear differential operators}

The results presented here are to be considered well-known. 
However, some are only to be found in the literature in a different setting than we need.

\subsection{Linear elliptic operators}

Let $E, F \to M$ be vector bundles equipped with a positive definite metric. We start by the classical \enquote{Fredholm alternative} for elliptic operators on closed manifolds.

\begin{lemma}[Fredholm alternative on closed manifolds] \label{le: Fredholm alternative closed mfld} \index{Fredholm alternative}
Assume that $M$ is a closed manifold, $k \in \R$ and
\[
	P: H^{k+m}(M, E) \to H^k(M, F)
\]
is a differential operator of order $m$ with injective principal symbol. Then
\begin{equation} \label{eq: fredholm_alt}
	H^k(M, F) = \im(P) \oplus \ker(P^*),
\end{equation}
where $P^*$ is the formal adjoint as an operator
\[
	P^*: H^k(M, F) \to H^{k-m}(M, E).
\]
Extend or restrict $P$ and $P^*$ to act on the spaces
\begin{align*}
	\tilde P&: H^{-k+m}(M, E) \to H^{-k}(M, F), \\
	\tilde P^*&: H^{-k}(M, F) \to H^{-k-m}(M, E).
\end{align*}
Then $\im(\tilde P)$ is the annihilator of $\ker(P^*)$ and $\ker(\tilde P^*)$ is the annihilator of $\im(P)$ under the isomorphism $H^{-k}(M, E) \cong H^{k}(M, E)'$.

In particular, if $k\geq 0$, the sum in \eqref{eq: fredholm_alt} is $L^2$-orthogonal. In case $k = \infty$, equation \eqref{eq: fredholm_alt} holds true.
\end{lemma}
\begin{proof}
See for example \cite{Besse1987}*{Appendix I} for equation \eqref{eq: fredholm_alt} when $k \geq 0$. Generalising this to any $k \in \R$ is straightforward, when using that
\begin{align*}
	H^{-k}(M, E) &\to H^k(M, E)' \\
	f &\mapsto (\varphi \mapsto \langle D^{-k}f, D^k \varphi \rangle_{L^2(M ,E)})
\end{align*}
is an isomorphism. 
\end{proof}

One part of the previous lemma generalises to non-compact manifold. 

\begin{lemma}
Let $M$ be a possibly non-compact manifold and let $K \subset M$ be a compact subset and let $k \in \R \cup \{\infty\}$. Assume that
\[
	P: H^{k+m}_K(M, E) \to H^k_K(M, F)
\]
is a differential operator of order $m$ with injective principal symbol. Assume furthermore that $P$ is injective. Then
\[
	\im(P) \subset H^k_K(M, F)
\]
is closed and $P$ is an isomorphism of Hilbert spaces onto its image. 
\end{lemma}
\begin{proof}
By \cite{BaerWafo2014}*{Sec. 1.6.2.}, we can embed an open neighbourhood $U$ of $K$ isometrically into a closed Riemannian manifold $(K', \g')$. Denote the embedding by $\iota: U \hookrightarrow K'$. Moreover, we can extend the vector bundles in a smooth way. Let us for simplicity still denote them by $E$ and $F$. For any section $f:M \to E$, define $\iota_*f :K' \to E$ such that $f|_K = (\iota_*f) \circ\iota|_K$, just by multiplying by a bump function which equals $1$ on $K$ and vanishes outside $U$. It follows that there is a differential operator with injective principal symbol
\[
	Q: H^{k+m}(K', E) \to H^{k}(K', F)
\]
such that the following diagram commutes:
\[
    \xymatrix{
       H^{k+m}_K(M,E) \ar[r]^{P} \ar[d]_{\iota_*} & H^k_K(M, F) \ar[d]^{\iota_*} \\
       H^{k+m}(K', E) \ar[r]^{Q}  & H^{k}(K', F) }.
\]
Choose a function $\lambda: K' \to \R$ such that $\lambda(x) > 0$ for all $x \in K' \backslash \iota(K)$ and $\lambda|_{\iota(K)} = 0$. We claim that 
\[
	Q^*Q + \lambda: H^{k+m}(K', E) \to H^{k-m}(K', E)
\]
is an isomorphism of Hilbert spaces (in the smooth case, $k=\infty$, we claim that this is an isomorphism of Fr{\'e}chet spaces). By Lemma \ref{le: Fredholm alternative closed mfld}, it suffices to show that $\ker(Q^*Q + \lambda) = \{0\}$, since $Q^*Q + \lambda$ is formally self-adjoint. For any $a \in \ker(Q^*Q + \lambda)$ it follows that $a$ is smooth and
\[
	\int_{K'} \abs{Qa}^2 + \lambda \abs{a}^2 dVol  = 0.
\]
Hence $\supp(a) \subset \iota(K)$ and $Qa = 0$. This implies that $b := \iota^*a$, extended to whole $M$ by zero, solves $P(b) = 0$. Since $\supp(b) \subset K$ and $P$ is injective, this implies that $b = 0$ and hence $a = 0$. We conclude the claim.

Assume now that $P(u_n) \to f$ in $H^k_K(M, F)$, with $u_n \in H^{k+m}_K(M, E)$. It follows that $Q(\iota_* u_n) \to \iota_* f$ in $H^k_{\iota(K)}(K', F)$ and $\iota_* u_n \in H^{k+m}_{\iota(K)}(K', E)$. Hence
\[
	(Q^*Q + \lambda)(\iota_* u_n) = Q^*Q(\iota_* u_n) \to Q^*(\iota_*(f))
\]
in $H^{k-m}_{\iota(K)}(K', E)$. Therefore, there is a $v \in H^{k+m}(K', E)$ such that $\iota_* u_n \to v$ in $H^{k+m}(K', E)$. Since $\supp(\iota_* u_n) \subset \iota(K)$ and $\iota_* u_n \to v$ as distributions, the support of $v$ cannot be larger than $\iota(K)$. Hence $v \in H^k_{\iota(K)}(K', E)$. Now define
\[
	u := \iota^* v \in H^{k+m}_K(U, E)
\]
and extend it by zero to an element in $H^{k+m}_K(M, E)$. Note that $u_n \to u$ in $H^{k+m}_K(M, E)$. It follows that 
\[
	P(u) = \lim_{n \to \infty} P(u_n) = f,
\]
as claimed (in the case $k = \infty$, the last line is to be thought of as a limit of a net).
\end{proof}

\begin{definition}[Laplace type operators] \index{laplace type operator}
Let $g$ be a Riemannian metric on $M$.
A differential operator $P \in \mathrm{Diff}_2(E, E)$ is called a Laplace type operator if its principal symbol is given by the metric. 
Equivalently, in local coordinates, $P$ takes the form
\[
	P = - \sum_{i,j} g^{ij}\frac{\partial^2}{\partial x^i \partial x^j} + l.o.t.
\]
\end{definition}

We will need the following theorem, known as the \emph{Strong unique continuation property}. We quote the statement from \cite{Baer1997}. For a proof, see \cite{Aronszajn1957}*{Thm. on p. 235 and Rmk. 3 on p. 248}.

\begin{thm}[Aronszajn's Unique Continuation Theorem] \label{thm: aronszajn}
Let $(M,g)$ be a connected Riemannian manifold and let $P$ be a Laplace type operator acting on sections of a vector bundle $E \to M$. Assume that $Pu = 0$ and that $u$ vanishes at some point of infinite order, i.e. that all derivatives vanish at that point. Then $u = 0$. 
\end{thm}

\begin{cor} \label{cor: Laplace-type closed image}
Let $k \in \R \cup \{\infty\}$. Assume that $M$ is connected. Let $K \subset M$ be a compact subset such that $K \neq M$. Assume that 
\[
	P: H^{k+2}_K(M, E) \to H^k_K(M, E)
\]
is a Laplace-type operator. Then
\[
	\im(P) \subset H^k_K(M, E)
\]
is closed and $P$ is an isomorphism of Hilbert spaces (Fr{\'e}chet spaces if $k = \infty$) onto its image. 
\end{cor}
\begin{proof}
We only need to show that $P$ is injective. Assume that $Pu= 0$. Since $u|_{M \backslash K} = 0$, Theorem \ref{thm: aronszajn} implies that $u = 0$. 
\end{proof}

\subsection{Linear wave equations} \label{sec: Waves} \index{wave!equation}

In the literature, there are many variants of stating the well-posedness of the Cauchy problem for linear wave equations with initial data of Sobolev regularity. The statement that is relevant for our purposes is not in the form we need it in the literature, but can be derived by standard techniques. 

\begin{definition}[Wave operator] \index{wave!operator}
Let $g$ be a Lorentzian metric on $M$. A differential operator $P \in \mathrm{Diff}_m(E, E)$ is called a wave operator if its principal symbol is given by the metric.
Equivalently, in local coordinates, $P$ takes the form
\[
	P = - \sum_{i,j} g^{ij}\frac{\partial^2}{\partial x^i \partial x^j} + l.o.t.
\]
\end{definition}

Wave operators are sometimes also called \emph{normally hyperbolic operators}. 
We assume here that $(M, g)$ is a globally hyperbolic spacetime and let $\S \subset M$ be a Cauchy hypersurface and $t:M \to \R$ a Cauchy temporal function such that $\S = t^{-1}(t_0)$ for some $t_0 \in t(M)$.
Let $E \to M$ be a real vector bundle and let $P$ be a wave operator acting on sections in $E$. Denote by $\nu$ the future pointing unit normal vector field on $\S$.

\begin{thm}[Existence and uniqueness of solution] \label{thm: WellposednessLinearWaves} \index{Cauchy problem!wave equation}
Let $k \in \R \cup \{\infty\}$ be given. For each $(u_0, u_1, f) \in H_{loc}^k(\S, E|_\S) \oplus H_{loc}^{k-1}(\S, E|_\S) \oplus CH^{k-1}_{loc}(M, E, t)$, there is a unique $u \in CH_{loc}^k(M, E, t)$ such that 
\begin{align*}
	Pu &= f, \\
	u|_\S &= u_0, \\
	\na_\nu u|_\S &= u_1.
\end{align*}
Moreover, we have finite speed of propagation, i.e.  \index{finite speed of propagation!wave equations}
\[
	\supp(u) \subset J \left(\supp(u_0) \cup \supp (u_1) \cup K \right),
\]
for any subset $K \subset M$ such that $\supp(f) \subset J(K)$.
\end{thm}

The theorem is proven by a standard method, translating the result in (\cite{BaerWafo2014}*{Thm. 13} and \cite{BaerGinouxPfaeffle2007}*{Thm. 3.2.11}), where spatially compact support was assumed, to the general case.
This can be done due to finite speed of propagation for wave equation.
It was done in the smooth case when $k = \infty$ in \cite{BaerFredenhagen2009}*{Cor. 5 in ch. 3}. 
Let us remark that from \cite{BaerWafo2014}*{Thm. 13} we only conclude that $u \in C^0(I, H^k_{loc}) \cap C^1(I, H^{k-1}_{loc})$, but since we assume more regularity on the right hand side $f$, we can use the equation $Pu = f$ to conclude the stated regularity on $u$. A simple corollary is the following.

\begin{cor}[Continuous dependence on initial data] \label{cor: cont_dep_id}
Let $k \in \R \cup \{\infty\}$ be given. Then the map
\begin{align*}
	CH^k_{loc}(M, E, t) \cap \ker(P) &\to H_{loc}^k(\S, E|_\S) \oplus H_{loc}^{k-1}(\S, E|_\S) \\
	u &\to (u|_\S, \na_\nu u|_\S)
\end{align*}
is an isomorphism between topological vector spaces. In particular, the inverse map is continuous.
\end{cor}
\begin{proof}[Proof of the corollary]
By the preceding theorem, this map is continuous and bijective between Fr{\'e}chet spaces. The open mapping theorem for Fr{\'e}chet spaces implies the statement.
\end{proof}

\end{sloppypar}
\end{document}